\pdfoutput=1
\documentclass[11pt, letterpaper]{article}
\usepackage{blindtext}
\usepackage{hyperref}
\hypersetup{
	colorlinks=true,
	linkcolor=blue!70!black,
	citecolor=blue!70!black,
	urlcolor=blue!70!black
}

\usepackage[english]{babel}
\usepackage[T1]{fontenc}

\usepackage[margin=1in]{geometry}

\usepackage{amsmath}
\usepackage{amssymb}
\usepackage{amsthm}
\usepackage{booktabs}
\usepackage{algorithm}
\usepackage[lined,boxed,ruled,norelsize,algo2e,linesnumbered]{algorithm2e}
\usepackage{bbm}
\usepackage{bbold}
\usepackage{parskip}
\usepackage{graphicx}
\graphicspath{{./figs/}}

\usepackage{cleveref}
\usepackage{thm-restate}

\newtheorem{theorem}{Theorem}
\newtheorem{lemma}{Lemma}
\newtheorem{fact}{Fact}
\newtheorem{definition}{Definition}
\newtheorem{corollary}{Corollary}
\newtheorem{proposition}{Proposition}

\newtheorem{problem}{Problem}
\newtheorem{observation}{Observation}
\newtheorem{remark}{Remark}

\newcommand{\defeq}{:=}
\newcommand{\norm}[1]{\left\lVert#1\right\rVert}
\newcommand{\norms}[1]{\lVert#1\rVert}
\newcommand{\inprod}[2]{\left\langle#1, #2\right\rangle}
\newcommand{\inprods}[2]{\langle#1, #2\rangle}
\newcommand{\eps}{\epsilon}
\newcommand{\lam}{\lambda}
\newcommand{\Lam}{\Lambda}

\newcommand{\argmin}{\textup{argmin}} 
\newcommand{\R}{\mathbb{R}}

\newcommand{\N}{\mathbb{N}}

\newcommand{\half}{\frac{1}{2}}

\newcommand{\1}{\mathbbm{1}}
\newcommand{\0}{\mathbb{0}}
\newcommand{\E}{\mathbb{E}}

\newcommand{\Nor}{\mathcal{N}}

\newcommand{\xset}{\mathcal{X}}

\newcommand{\ma}{\mathbf{A}}

\newcommand{\id}{\mathbf{I}}

\usepackage{xcolor}
\definecolor{burntorange}{rgb}{0.8, 0.33, 0.0}
\newcommand{\kjtian}[1]{\textcolor{burntorange}{\textbf{kjtian:} #1}}

\newcommand{\arun}[1]{\textcolor{red}{\textbf{arun:} #1}}

\newcommand{\tO}{\widetilde{O}}

\newcommand{\Par}[1]{\left(#1\right)}
\newcommand{\Brack}[1]{\left[#1\right]}
\newcommand{\Brace}[1]{\left\{#1\right\}}

\newcommand{\alg}{\mathcal{A}}
\newcommand{\oracle}{\mathcal{O}}
\newcommand{\ball}{\mathbb{B}}
\newcommand{\set}{\mathcal{K}}
\newcommand{\dd}{\textup{d}}
\newcommand{\mb}{\mathbf{B}}
\newcommand{\mm}{\mathbf{M}}

\newcommand{\obo}{\oracle_{\textup{bo}}}
\newcommand{\bx}{\bar{x}}
\newcommand{\xout}{x_{\mathrm{out}}}
\newcommand{\cba}{C_{\textup{ba}}}
\newcommand{\BA}{\mathsf{BallAccel}}
\newcommand{\lams}{\lam_\star}
\newcommand{\xs}{x^\star}
\newcommand{\frl}{f_{\rho, \lam}}
\newcommand{\frlbx}{f_{\rho, \lam, \bar{x}}}
\newcommand{\CN}{\mathsf{ConstrainedNewton}}
\newcommand{\UCS}{\mathsf{UnconstrainedSGD}}
\newcommand{\UCSS}{\mathsf{UnconstrainedSGDConv}}
\newcommand{\codeStyle}[1]{{\bfseries #1} }
\newcommand{\codeInput}{\codeStyle{Input:}}	
	
\newcommand{\codeReturn}{\codeStyle{Return:}}	
	
\newcommand{\xslmz}{\xs_{\lam, \mu, z}}
\newcommand{\xslzv}{\xs_{\Lam, z, v}}
\newcommand{\xavg}{x_{\textup{avg}}}
\newcommand{\ms}{\mathbf{S}}
\newcommand{\normop}[1]{\left\lVert#1\right\rVert_{\textup{op}}}
\newcommand{\eadd}{\eps_{\textup{add}}}
\newcommand{\emul}{\eps_{\textup{mul}}}
\newcommand{\mh}{\mathbf{H}}
\newcommand{\tg}{\tilde{g}}
\newcommand{\tmh}{\widetilde{\mh}}
\newcommand{\xsa}{x^\star_\alpha}
\SetKwComment{Comment}{$\triangleright$\ }{}
\SetCommentSty{mycommfont}

\newcommand{\timeSymbol}{\mathcal{T}}
\newcommand{\depthSymbol}{\mathcal{D}}
\newcommand{\tQuery}{\timeSymbol_{\mathrm{query}}}
\newcommand{\dQuery}{\depthSymbol_{\mathrm{query}}}
\newcommand{\poly}{\textup{poly}}
\newcommand{\calT}{\mathcal{T}}
\newcommand{\calD}{\mathcal{D}}

\title{Closing the Computational-Query Depth Gap \\ in Parallel Stochastic Convex Optimization}
\author{	Arun Jambulapati\\
            University of Michigan\\
            \texttt{jmblpati@gmail.com}
	\and
	    Aaron Sidford\\
            Stanford University\\
            \texttt{sidford@stanford.edu}
	\and
            Kevin Tian\\
            University of Texas at Austin\\
            \texttt{kjtian@cs.utexas.edu}
}

\date{}
\begin{document}

\maketitle

\begin{abstract}

We develop a new parallel algorithm for minimizing Lipschitz, convex functions with a stochastic subgradient oracle. The total number of queries made and the query depth, i.e., the number of parallel rounds of queries, match the prior state-of-the-art, \cite{CarmonJJLLST23}, while improving upon the computational depth by a polynomial factor for sufficiently small accuracy. When combined with previous state-of-the-art methods our result closes a gap between the best-known query depth and the best-known computational depth of parallel algorithms.

Our method starts with a \emph{ball acceleration} framework of previous parallel methods, i.e., \cite{CarmonJJJLST20, AsiCJJS21}, which reduce the problem to minimizing a regularized Gaussian convolution of the function constrained to Euclidean balls. By developing and leveraging new stability properties of the Hessian of this induced function, we depart from prior parallel algorithms and reduce these ball-constrained optimization problems to stochastic unconstrained quadratic minimization problems. Although we are unable to prove concentration of the asymmetric matrices that we use to approximate this Hessian, we nevertheless develop an efficient parallel method for solving these quadratics. Interestingly, our algorithms can be improved using fast matrix multiplication and use nearly-linear work if the matrix multiplication exponent is 2.

\end{abstract}

\tableofcontents

\newpage

\section{Introduction}
\label{sec:intro}

Consider the classic problem of \emph{Lipschitz convex optimization}. In this problem, there is a convex $f : \R^d \to \R$ that is $1$-\emph{Lipschitz}, i.e., 
$|f(x) - f(y)| \leq \norm{x - y}$ for all $x,y\in \R^d$, that is guaranteed to have a minimizer $\xs \in \R^d$ with $\norm{\xs} \leq 1$. The goal of the problem is to compute an (expected) \emph{$\epsilon$-approximate minimizer} to $f$, i.e., $x \in \R^d$ with $\E f(x) \leq f(\xs) + \epsilon$ given access to $f$ only though a subgradient oracle $g$ that when queried at $x \in \R^d$ outputs a vector $g(x) \in \partial f(x)$, where $\partial f$ is the set of subgradients of $f$ at $x$. We focus on this standard setting in the introduction for simplicity, however our results extend to the more general case of bounded stochastic gradient oracles and further relaxations of the bounds on Lipschitz continuity and the minimizer (see \Cref{prob:sco}).

Lipschitz convex optimization is foundational in optimization theory, and its study has motivated well-known optimization algorithms. Simple, classic subgradient descent solves the problem with $O(\eps^{-2})$ oracle queries \cite{NemirovskiY83}, and cutting plane methods solve the problem with $O(d \log \eps^{-1})$ oracle queries \cite{KTE88}. Consequently, the query complexity of the problem, i.e., the number of queries needed to solve the problem in the worst case, is $O(\min\{\epsilon^{-2},d \log(\epsilon^{-1})\}$. Furthermore, this bound is known to be optimal among deterministic algorithms for all settings of $\epsilon$ and $d$ \cite{NemirovskiY83}, and is optimal even among randomized and quantum algorithms in certain settings \cite{AgarwalBRW12, GargKNS21}.

Due to the massive growth in dataset sizes and use of parallel computing resources, a line of work has studied \emph{parallel} variants of Lipschitz convex optimization \cite{Nem94, DuchiBM12, BalkanskiS18, BubeckJLLS19, CarmonJJLLST23} and non-Euclidean generalizations \cite{DG19,ChakrabartyGJS23}. Study of this problem dates to at least \cite{Nem94} which proposed the \emph{parallel oracle access model}, in which the algorithm proceeds in $T$ rounds and in round $t \in [T]$, the algorithm queries the oracle with $n_t$ points $x_{t,1},\ldots,x_{t,n_t} \in \R^d$ and receives the output of the oracle on each point.
In round $t$, the $n_t$ queried points can depend only on the queries in the previous rounds and the output of the oracle in those rounds (and additional randomness used by the algorithm). We call such an algorithm \emph{highly parallel} \cite{BubeckJLLS19} if 
the number of queries in each round is bounded by a polynomial in $d$ and a natural condition number for the problem, e.g., $n_t = \mathrm{poly}(d,\epsilon^{-1})$ for all $t \in [T]$. The total number of rounds of the algorithm, $T$, is called the \emph{query depth} of the algorithm, and is a natural measure of its parallel performance.

Perhaps surprisingly, nontrivial parallel speedups, i.e., parallel algorithms whose query depth is better than the best-known query complexity, are only known for certain $\epsilon$ ranges. In fact, the $O(\epsilon^{-2})$ complexity of simple subgradient descent is optimal among highly parallel algorithms for sufficiently large $\epsilon$. The associated lower bound was shown for all $\epsilon \gtrsim d^{-1/6}$ by \cite{Nem94, BalkanskiS18} and for all $\epsilon \gtrsim d^{-1/4}$ by \cite{BubeckJLLS19}.\footnote{We use $\lesssim$, $\gtrsim$, and $\tO$ to hide polylogarithmic factors in $d$ and $\epsilon^{-1}$ in the introduction, and more broadly we use this notation to hide polylogarithmic factors in $d$ and $\frac{LR}{\eps}$ throughout the paper in the context of Problem~\ref{prob:sco}.} Additionally, when $\eps \lesssim d^{-1}$, the current state-of-the-art query depth is achieved by applying classical cutting plane methods, e.g., \cite{Vaidya96}. However, in the regime where $d^{-1/4} \gtrsim \epsilon \gtrsim d^{-1}$, which we term the \emph{intermediate regime} of $\eps$, nontrivial parallel speedups are known and there are algorithms which improve upon both subgradient descent and cutting plane methods. Namely, Lipschitz convex optimization was first shown to be solvable with query depth $\tO(d^{1/4}\eps^{-1})$ by \cite{DuchiBM12}, and then $\tO(d^{1/3} \eps^{-2/3})$ by \cite{BubeckJLLS19}, the current state-of-the-art query depth.

However, beyond query depth, there are other natural ways to parameterize the complexity of a highly parallel algorithm. %
Specifically,  we measure the complexity of parallel algorithms as follows.

\begin{definition}[Parallel complexity]\label{def:parallel_alg}
We define the following four properties of an algorithm solving an optimization problem, e.g., Problem~\ref{prob:sco}, with parallel access to an oracle $g$. 
\begin{enumerate}
    \item \emph{Query depth}: number of sequential rounds of interaction with $g$ (queries submitted in batch).
    \item \emph{Query complexity}: total number of queries to $g$.
    \item \emph{Computational depth}: number of sequential rounds of computation, outside of querying $g$.
    \item \emph{Computational complexity}: amount of computational work performed, outside of querying $g$. 
\end{enumerate}
If $g$
can be implemented with $O(\tQuery)$ work and $O(\dQuery)$ depth, we write that an algorithm can be implemented with $O(a \cdot \dQuery + b)$ depth and $O(c \cdot \tQuery + d)$ work when its query depth is $O(a)$, query complexity is $O(c)$, computational depth is $O(b)$, and computational complexity is $O(d)$.
\end{definition}

Recently, \cite{CarmonJJLLST23} designed an algorithm which matched the $\tO(d^{1/3} \eps^{-2/3})$ query depth of \cite{BubeckJLLS19}, while simultaneously achieving a \emph{query complexity} of $\tO(d^{1/3} \eps^{-2/3} + \eps^{-2})$. This query complexity improved upon the $\tO(d^{4/3}\eps^{-8/3})$ query complexity of  \cite{BubeckJLLS19} and, when $\eps \lesssim d^{-1/4}$, matched that of subgradient descent, which is optimal for $\epsilon \gtrsim d^{-1/2}$ (as discussed earlier).  Unfortunately, the \emph{computational depth} of \cite{CarmonJJLLST23}  is $\tO(d^{1/4}\eps^{-1})$  (matching that of \cite{DuchiBM12}). This computational depth scales polynomially worse than the query depth of \cite{CarmonJJLLST23} for $\epsilon \lesssim d^{-1/4}$, and is larger than the computational depth of state-of-the-art cutting plane methods, e.g., \cite{Vaidya96}, when $\epsilon \lesssim d^{-3/4}$. 

The key question motivating our work is whether this gap between the computational and query depths of state-of-the-art parallel algorithms in the intermediate regime is inherent. Specifically we address an open problem left by \cite{CarmonJJLLST23} as to whether there is an algorithm which, in the intermediate regime of $\eps$, obtains the best-known query depth and query complexity, while simultaneously obtaining a computational depth no worse than its query depth (ideally at low overhead to the algorithm's computational complexity). Closing this gap is a natural problem that would expand the theory for parallel convex and stochastic optimization, and potentially be of broader utility.

\paragraph{Our results.}
Our main result is a new algorithm which closes this gap for Lipschitz convex optimization and, more broadly, for stochastic convex optimization, as stated in \Cref{prob:sco}. The most general form of our result, Theorem~\ref{thm:main_formal}, is stated in Section~\ref{sec:framework}. For simplicity in the introduction, we state the specialization of Theorem~\ref{thm:main_formal} to Lipschitz convex optimization here.

\begin{theorem}\label{thm:main}
If queries to a subgradient oracle are implementable with	$O(\dQuery)$ depth and $O(\tQuery)$ work, then there is a randomized algorithm which solves Lipschitz convex optimization with
\[
\tO(d^{\frac 1 3}\epsilon^{- \frac 2 3} \cdot \dQuery + d^{\frac 1 3}\epsilon^{- \frac 2 3})
\text{ depth and }
\tO\Par{(d^{\frac 1 3} \epsilon^{- \frac 2 3} + \epsilon^{-2})\cdot \tQuery + 
d^{\frac 4 3}\epsilon^{- \frac 2 3} + d^{\frac{5 - \omega} 3}\epsilon^{- \frac {4\omega - 2} 3} 
}
\text{ work,}
\]
where $\omega < 2.372$ \cite{AlmanDWXXZ24} is such that multiplying $d \times d$ matrices requires $O(d^\omega)$ work.
\end{theorem}

Theorem~\ref{thm:main} matches the query depth and query complexity of the state-of-the-art algorithm  \cite{CarmonJJLLST23} (in terms of query depth) in the intermediate regime and attains a computational depth matching its query depth up to logarithmic factors (see \Cref{tab:sco} for a more complete comparison to prior work). Interestingly, the method uses fast matrix multiplication, a technique used to obtain state-of-the-art work and depth complexities for linear system solving; however, if $\omega = 2$, assuming vector operations using a stochastic gradient oracle require $\Omega(d)$ work, then its computational complexity is no worse than the work to query the oracle. Moreover, if $\tQuery$ is moderately larger than $d$ (e.g., $\tQuery = \Omega(d \cdot \eps^{-\half})$ for the current value of $\omega \neq 2$), the overhead of $d^{\frac{5 - \omega} 3}\epsilon^{- \frac {4\omega - 2} 3} $ is a low-order term compared to $\eps^{-2}\cdot \tQuery$. In the more general Corollary~\ref{cor:tradeoff_main} later in the paper, we show that it is possible obtain different tradeoffs between computational complexity and computational depth.

Beyond these quantitative improvements to parallel Lipschitz and stochastic convex optimization, to obtain our results, we provide several insights on related tools of potential independent interest (all outlined in \Cref{sec:overview}). First, we provide a structural result (Lemma~\ref{lem:conv_stable}) about Gaussian convolutions of convex functions, a central tool in stochastic optimization. When combined with prior parallel optimization machinery, Lemma~\ref{lem:conv_stable} reduces our problem to solving certain structured stochastic quadratic optimization problems in parallel. We then provide a new parallel algorithm for solving these quadratic optimization problems, which circumvents the need for concentration bounds. Along the way, we provide tools for boosting expected optimality bounds into high probability and handling hard constraints. We hope these tools may find broader use in optimization and learning theory and facilitate reaping the rewards of parallelism while mitigating the computational costs.

\begin{table}
    \centering
    \renewcommand{\arraystretch}{1.75}
    \begin{tabular}{{c}{c}{c}{c}}
    \toprule
      Method   &  Query depth & Query complexity & Computational depth \\
      \midrule
       SGD~\cite{NemirovskiY83}  & $\epsilon^{-2}$ & $\epsilon^{-2}$ & $\epsilon^{-2}$\\
       \cite{DuchiBM12} & $d^{\frac{1}{4}}\epsilon^{-1}$ & $d^{\frac{1}{4}}\epsilon^{-1} + \epsilon^{-2}$ & $d^{\frac{1}{4}}\epsilon^{-1}$\\
       \cite{BubeckJLLS19} & $d^{\frac{1}{3}}\epsilon^{-\frac{2}{3}}$ & $d^{\frac{4}{3}}\epsilon^{-\frac{8}{3}}$ & $d^{\frac{4}{3}}\epsilon^{-\frac{8}{3}}$ \\
       \cite{CarmonJJLLST23} & $d^{\frac 1 3}\epsilon^{-\frac 2 3}$ & $d^{\frac 1 3}\epsilon^{-\frac 2 3} + \epsilon^{-2}$ & $d^{\frac 1 3}\epsilon^{-\frac 2 3} + d^{\frac 1 4}\epsilon^{-1}$\\
       CPM*~\cite{Vaidya96} & $d$ & $d$ & $d$ \\
       \midrule
       Theorem~\ref{thm:main} & $d^{\frac{1}{3}}\epsilon^{-\frac{2}{3}}$ & $d^{\frac{1}{3}}\epsilon^{-\frac{2}{3}}+\epsilon^{-2}$ & $d^{\frac{1}{3}}\epsilon^{-\frac{2}{3}}$ \\
       \bottomrule
    \end{tabular}
    \caption{\textbf{Highly parallel Lipschitz convex optimization algorithms.} The table depicts the history of improvements for solving Lipschitz convex optimization algorithms hiding polylogarithmic factors in $d$ and $\eps^{-1}$. CPM* refers to ``cutting plane methods'' and to the best of the authors knowledge \cite{Vaidya96} is the first paper to achieve the state-of-the-art complexity stated in the table; there are additional CPM results discussed in \Cref{rem:cutting_plane}. The table applies to Problem~\ref{prob:sco} if each occurrence of $\epsilon^{-1}$ is replaced with $\kappa \defeq \tfrac{LR}{\eps}$ and the CPM* line is changed (again, see \Cref{rem:cutting_plane}).
    }
    \label{tab:sco}
\end{table}

\begin{remark}[Parallel complexity of cutting plane methods]
\label{rem:cutting_plane}	
Cutting plane methods have a longer history than that conveyed in the CPM* line of \Cref{tab:sco}. \cite{levin1965algorithm,newman1965location} showed that it is possible to obtain query complexity $\tO(d)$ and since then a line of work has established different tradeoffs between query complexity and computational complexity \cite{shor1977cut,yudin1976informational,khachiyan1980polynomial,khachiyan1988method,nesterov1989self,Vaidya96,bertsimas2004solving,lee2015faster,jiang2020improved}. Though computational depth was not necessarily highlighted in these works, we believe the first method with $\tO(d)$ query complexity that could be implemented in depth $\tO(d)$ is due to \cite{Vaidya96}; subsequent papers may or may not have the same property. For stochastic convex optimization (\Cref{prob:sco}), though not explicitly stated, we believe the state-of-the-art is to leverage \cite{Vaidya96} within the framework of \cite{SidfordZ23} to obtain an algorithm query and computational depth $\tO(d)$ and query complexity $\tO(d \cdot \poly(\kappa))$ for $\kappa \defeq \tfrac{LR}{\eps}$.
\end{remark}

\paragraph{Paper organization.}
We assemble the pieces to prove Theorem~\ref{thm:main} and its generalization, Theorem~\ref{thm:main_formal}, throughout the rest of the paper. In Section~\ref{sec:overview}, we overview our approach. We first review facts about Gaussian convolutions and a ball acceleration result from \cite{CarmonJJLLST23}, which constitutes our main algorithm framework. Additionally, we prove a result about the stability of Hessians of Gaussian convolutions of convex functions, which is the main structural insight enabling our new algorithms. In Section~\ref{sec:quadratic}, we give an efficient parallel algorithm for optimizing (suitably regularized) local quadratic approximations to a Gaussian convolution. Finally, in Section~\ref{sec:framework}, we show how to obtain our results by using this quadratic solver to implement the constrained ball oracles required by the acceleration framework by performing a binary search over our subroutine in Section~\ref{sec:quadratic}.

\paragraph{General notation.} For $d \in \N$, we let $\0_d$ and $\1_d$ denote the all-zeroes and all-ones vectors in $\R^d$, $\id_d$ denote the identity matrix in $\R^{d \times d}$, and $[d] \defeq \{i \in \N \mid 1 \le i \le d\}$.
 We let $\norm{\cdot}$ denote the Euclidean norm of a vector. For $x \in \R^d$ we let $\ball_x(r) \defeq \{x' \in \R^d \mid \norm{x' - x} \le r\}$ and $\ball(r) \defeq \ball_{\0_d}(r)$ when $d$ is clear from context. We use $\preceq$ to denote the Loewner partial ordering over $d \times d$ symmetric matrices, i.e., $\ma \preceq \mb$ if and only if $x^\top \ma x \leq x^\top \mb x$ for all $x$, and define $\succeq$ analogously. For a positive semidefinite (PSD) $\ma \in \R^{d \times d}$, we let $\norm{v}_{\ma} \defeq (v^\top \ma v)^{1/2}$ be the induced seminorm.

\paragraph{Parallel computation model.} We assume a parallel computation model where all vector operations in $\R^d$ (e.g., addition and scalar multiplication) require $O(d)$ work and $O(1)$ depth, and that all matrix-vector multiplications (including computing dot products) require $O(\log d)$ depth. We let $\omega < 2.372$ \cite{AlmanDWXXZ24} be defined such that two $d \times d$ matrices can be multiplied with work $O(d^\omega)$. By a known reduction (\cite{Pan87}; see also discussion in \cite{PanR85}), under this definition of $\omega$, matrix multiplication can be performed in work $O(d^\omega)$ and depth $O(\log d)$. Under different parallel models, some of these  bounds may incur polylogarithmic factor overheads.

%

%
\section{Technical overview}\label{sec:overview}

In this remainder of the paper we consider the following stochastic convex optimization problem.

\begin{problem}[Stochastic convex optimization]\label{prob:sco}
	In the \emph{stochastic convex optimization} problem we are given $\epsilon, L, R > 0$ and access to a  \emph{stochastic gradient oracle} $g: \R^d \to \R^d$ 
	satisfying, for all $x \in \R^d$, $\E g(x) \in \partial f(x)$ and $\E \|g(x)\|^2 \le L^2$ for convex $f: \R^d \to \R$. The goal is to output an \emph{expected $\eps$-approximate minimizer of $f$ over $\ball(R)$}, i.e., $\xout \in \R^d$ such that
	$\E f(\xout) \leq \min_{x \in \ball(R)} f(x) + \eps$. We assume $g$ can be implemented with $O(\tQuery)$ work and $O(\dQuery)$ depth.
\end{problem}

Note that stochastic convex optimization (\Cref{prob:sco}) generalizes the Lipschitz convex optimization problem defined in Section~\ref{sec:intro}. By Jensen's inequality and the convexity of $\norm{\cdot}^2$, $\E \|g(x)\|^2 \le L^2$ implies that $\norm{\E g(x)} \le L$ and consequently in \Cref{prob:sco} at every point $x$ there is a subgradient of norm at most $L$. This implies that $f$ is $L$-Lipschitz and thus Lipschitz convex optimization is the special case of \Cref{prob:sco} when $L = R = 1$, each output of the stochastic subgradient oracle is deterministic, and $f$ has a minimizer $\xs \in \R^d$ with $\norm{\xs} \leq 1$. 

Our main result is the following efficient parallel algorithm for solving \Cref{prob:sco}. This theorem immediately implies Theorem~\ref{thm:main} in the special case of Lipschitz convex optimization. 

\begin{restatable}{theorem}{restatemainformal}\label{thm:main_formal}
	There is an algorithm ($\BA$ in Proposition~\ref{prop:ball_accel}, using Proposition~\ref{prop:parallel_ball_oracle} as a ball optimization oracle) which solves Problem~\ref{prob:sco} using:
	\begin{gather*}
		O\Par{d^{\frac 1 3}\kappa^{\frac 2 3}\log^{\frac{13}{3}}\Par{d\kappa}\log\log\Par{d\kappa} \cdot \dQuery + d^{\frac 1 3}\kappa^{\frac 2 3}\log^{\frac{28}{3}}(d\kappa)}\text{ depth,}\\
		\text{and } O\Par{\Par{d^{\frac 1 3}\kappa^{\frac 2 3}\log^{\frac{10}3}(d\kappa) + \kappa^2\log^{\frac{19}{3}}(d\kappa)}\cdot \tQuery + d^{\frac 4 3}\kappa^{\frac 2 3}\log^{\frac {10} 3}(d\kappa)+ d^{\frac{5-\omega}{3}}\kappa^{\frac{4\omega - 2}{3}}\log^{\frac{19}{3}}(d\kappa)}  \text{ work,}
	\end{gather*}
	where $\omega < 2.372$ \cite{AlmanDWXXZ24} is the matrix multiplication exponent, and $\kappa \defeq \frac{LR}{\eps}$.
\end{restatable}

In the remainder of this technical overview we cover the main steps in proving Theorem~\ref{thm:main_formal}, and discuss key insights and tools developed along the way, of possible broader utility. First, in \Cref{ssec:framework} we provide an overview of the general framework used by both our parallel algorithm and prior work. In \Cref{sec:overview:hessian-stable} we then discuss the key structural insight about this framework that we make and leverage to depart and improve upon prior work. In \Cref{sec:intro:hess_opt} we then discuss our main subroutine that we develop to leverage this structural insight and then in \Cref{sec:overview:parallel-rank-1} we discuss implementing the subroutine in low computational depth to obtain our result. 

\subsection{Framework: convolutions and acceleration}
\label{ssec:framework}

All prior parallel improvements over subgradient descent for \Cref{prob:sco} in the intermediate regime follow a similar broad framework \cite{DuchiBM12, BubeckJLLS19, CarmonJJLLST23}. Each of these works considers a process for smoothing $f$, i.e., working with a smooth approximation, and each uses accelerated optimization methods, i.e., some form of momentum, for optimizing the smoothing of $f$. Where the methods differ is in what smoothing is used, what accelerated method is applied, and how the accelerated method is implemented. Our method follows the approach of  \cite{CarmonJJLLST23} which applies ball acceleration frameworks to optimize the Gaussian convolution of $f$. We begin by reviewing these techniques and highlighting their implication for parallel stochastic convex optimization.

\paragraph{Gaussian convolution.}

To solve \Cref{prob:sco} rather than directly optimizing $f$, following the approach of  \cite{DuchiBM12, BubeckJLLS19, CarmonJJLLST23},
we instead apply methods that optimize the Gaussian convolution of $f$, i.e., the function resulting from convolving $f$ with a Gaussian, as defined below.

\begin{definition}[Gaussian convolution]\label{def:gconv}
	For $f: \R^d \to \R$ and $\rho \ge 0$, we let $f_\rho$ denote the convolution of $f$ with $\Nor(\0_d, \rho^2 \id_d)$, the normal distribution on $\R^d$ with covariance $\rho^2 \id_d$ and mean $\0_d$. We use $\ast$ to denote convolution and $\gamma_p$ to denote the density function of $\Nor(\0_d, \rho^2 \id_d)$, so that $f_\rho = f \ast \gamma_\rho$ and 
	\[
	f_\rho(x) \defeq \E_{\xi \sim \Nor(\0_d, \rho^2 \id_d)}\Brack{f(x - \xi)} = \int_{\R^d} f(x - \xi) \gamma_\rho(\xi) \dd \xi \text{ for all } x \in \R^d\,.
	\]
\end{definition}

Working with $f_\rho$ offers a number of advantages: it is smooth, twice differentiable, and stochastic approximations to its gradient and Hessian can be computed by querying the stochastic gradient oracle at appropriately chosen random points. Additionally, it satisfies a new structural property we develop in Section~\ref{sec:overview:hessian-stable}. Formally, we recall the following facts from prior work \cite{DuchiBM12, BubeckJLLS19}, where by the Alexandrov theorem, the first and second derivatives are almost-everywhere defined in the third item. The fourth item in Fact~\ref{fact:gconv} is shown in the proof of Lemma 8 in \cite{BubeckJLLS19}.

\begin{fact}[Lemma 8, \cite{BubeckJLLS19}]\label{fact:gconv}
	For all convex, $L$-Lipschitz $f: \R^d \to \R$, $\rho \ge 0$, and $x \in \R^d$:
	\begin{enumerate}
		\item $f_\rho$ is convex, $L$-Lipschitz, twice-differentiable, and satisfies $\nabla^2 f_\rho(x) \preceq \frac{L}{\rho} \id_d$,
		\item $f(x) \le f_\rho(x) \le f(x) + L \rho \sqrt{d}$,
		\item $\nabla f_\rho(x) = \int_{\R^d} \nabla f(x - \xi) \gamma_\rho(\xi)\dd \xi$, and
		\item $\nabla^2 f_\rho(x) = \int_{\R^d} \nabla^2 f(x - \xi) \gamma_\rho(\xi) \dd\xi = \frac 1 {\rho^2} \int_{\R^d} \nabla f(x + \xi)\xi^\top \gamma_\rho(\xi) \dd \xi$.
	\end{enumerate}
\end{fact}

In light of Fact~\ref{fact:gconv}, we make the following observation.

\begin{observation}\label{obs:conv_suffices}
	In the context of Problem~\ref{prob:sco}, let $\rho \defeq \frac{\eps}{2L\sqrt{d}}$. If a point $\xout$ solves an instance of Problem~\ref{prob:sco} with $f \gets f_\rho$ and $\eps \gets \frac \eps 2$, then $\xout$ also solves Problem~\ref{prob:sco} with $f\gets f$ and $\eps \gets \eps$.
\end{observation}
\begin{proof}
	By the first item in Fact~\ref{fact:gconv}, it is valid to let $f$ be $f_\rho$ in an instance of Problem~\ref{prob:sco}. By the second item in Fact~\ref{fact:gconv}, letting $x^\star$ achieve $\min_{x \in \ball(R)} f(x)$,
	\[\E f(\xout) \le \E f_\rho(\xout) \le f_\rho(x^\star) + \frac \eps 2 \le f(x^\star) + \eps.\]
\end{proof}
In the rest of the paper, we consider a fixed instance of Problem~\ref{prob:sco}; our approach is to optimize $f$ by instead designing an algorithm for optimizing its Gaussian convolution $f_\rho$. In light of Observation~\ref{obs:conv_suffices}, unless specified otherwise, we fix $\rho \defeq \frac{\eps}{2L\sqrt d}$ throughout the rest of the paper.

\paragraph{Ball acceleration.}
To optimize the Gaussian convolution $f_\rho$, we leverage recent advances in accelerated proximal point algorithms, specifically a recent framework termed \emph{ball acceleration} \cite{CarmonJJJLST20, AsiCJJS21}. Similar strategies were employed by \cite{BubeckJLLS19} and \cite{CarmonJJLLST23}, which reduce solving Problem~\ref{prob:sco} to a smaller number of carefully-designed subproblems; in the case of \cite{CarmonJJLLST23}, the subproblem is to minimize a regularized approximation to $f_\rho$ over a small Euclidean ball (a \emph{ball optimization oracle}). We specifically use the following variant of ball acceleration from \cite{CarmonJJLLST23}.

\begin{definition}[Ball optimization oracle]\label{def:boo}
We say $\obo$ is a \emph{$(\phi, \lam, r)$-ball optimization oracle} for $F: \R^d \to \R$ if given $\bx \in \R^d$, $\obo$ returns $x \in \ball_{\bx}(r)$ with 
\[\E \Brack{F(x) + \frac \lam 2 \norm{x - \bx}^2} \le \min_{x' \in \ball_{\bx}(r)} \Brace{F(x) + \frac \lam 2 \norm{x' - \bx}^2} + \phi.\]
\end{definition}

\begin{proposition}[Proposition 2, \cite{CarmonJJLLST23}]\label{prop:ball_accel}
Let $F: \R^d \to \R$ be $L$-Lipschitz and convex, let $R > 0$, and let $x^\star \in \ball(R)$. There is an algorithm $\BA$ which takes parameters $r \in (0, R]$ and $\eps \in (0, LR]$ with the following guarantee. Define
\[\kappa \defeq \frac{LR}{\eps},\; K \defeq \Par{\frac R r}^{\frac 2 3},\;\text{ and } \lams \defeq \frac{\eps K^2}{R^2}\log^2(K).\]
For a universal constant $\cba > 0$, $\BA$ produces $x \in \R^d$ such that $\E F(x) \le F(x^\star) + \eps$. Letting $\calT(\phi, \lam, r) \ge d$ and $\calD(\phi, \lam, r) \ge \log(d)$ denote the work and depth used by a $(\phi, \lam, r)$-ball optimization oracle, the computational complexity of $\BA$ is:
\begin{gather*}
\cba K \log^3\Par{\frac{R\kappa}{r}} \calT\Par{\frac{\lams r^2}{\cba}, \frac{\lams}{\cba}, r} \\
+ \sum_{j \in [\lceil \log_2 K + \cba \rceil]} \cba 2^{-j}K\log\Par{\frac{R\kappa}{r}} \calT\Par{\frac{\lams r^2}{\cba 2^j} \log^{-2}\Par{\frac{R\kappa} r}, \frac{\lams}{\cba}, r},
\end{gather*}
and the depth of $\BA$ is:
\begin{gather*}
\cba K \log^3\Par{\frac{R\kappa}{r}} \calD\Par{\frac{\lams r^2}{\cba}, \frac{\lams}{\cba}, r} \\
+ \sum_{j \in [\lceil \log_2 K + \cba \rceil]} \cba 2^{-j}K\log\Par{\frac{R\kappa}{r}} \calD\Par{\frac{\lams r^2}{\cba 2^j} \log^{-2}\Par{\frac{R\kappa} r}, \frac{\lams}{\cba}, r}.
\end{gather*}
\end{proposition}
 
 \paragraph{Implication and approach.}
Proposition~\ref{prop:ball_accel} allows us to focus on designing ball optimization oracles for $f_\rho$ in the remainder of the paper. Concretely, in our algorithm each oracle approximately solves a problem of the form
\begin{equation}\label{eq:frldef}
\min_{x \in \ball_{\bx}(r)} \frlbx(x) \defeq f_\rho(x) + \frac \lam 2 \norm{x - \bx}^2,
\end{equation}
where $f_\rho$ is the convolution of the density function of $\Nor(0, \rho^2 \id_d)$ and the function of interest $f$, and $\lam > 0$ is a regularization parameter. The key challenge we address is how to implement a ball oracle for the Gaussian convolution efficiently in parallel. Since pioneering work of \cite{DuchiBM12} it has been observed that stability of the Gaussian convolution can be useful in this endeavor, because higher moments of the Gaussian convolution can be efficiently approximated via parallel queries. \cite{DuchiBM12} leveraged smoothness (stability of the gradient) in accelerated gradient descent, \cite{BubeckJLLS19} leveraged higher-order smoothness and concentration to approximate the Gaussian convolution over a large regions, and \cite{CarmonJJLLST23} leveraged how it is possible to obtain stochastic gradients for the Gaussian convolution at one point by sampling the stochastic gradient oracle a nearby point. 

Our work exploits a new stability property of Gaussian convolutions (Lemma~\ref{lem:conv_stable}) that we introduce in the next \Cref{sec:overview:hessian-stable}. We leverage this property to essentially reduce implementing a ball oracle to solving a linear system induced by the Hessian of the Gaussian convolution, which we discuss how to do efficiently in parallel in \Cref{sec:intro:hess_opt,sec:overview:parallel-rank-1}. Along the way, we introduce several tools which may be of broader interest to the stochastic optimization theory community.

\subsection{Hessian stability of the Gaussian convolution} 
\label{sec:overview:hessian-stable}

Our starting point for designing our ball optimization oracle is the observation that the Hessian of the convolved objective $f_\rho$ is stable, in a precise sense, over balls of small radii $r \ll \rho$. To explain, note that for any $x, y \in \R^d$,
\[\nabla^2 f_\rho(x) = \int_{\R^d} \nabla^2 f(x - \xi) \gamma_\rho(\xi) \dd \xi = \int_{\R^d} \nabla^2 f(y - \xi) \gamma_\rho(x - y + \xi) \dd \xi. \]
So, if $\gamma_\rho(x - y + \xi) \approx \gamma_\rho(\xi)$ multiplicatively for all $\xi$, then similarly $\nabla^2 f_\rho(x) \approx \nabla^2 f_\rho(y)$. Unfortunately, this is not true: directly expanding shows
\[\frac{\gamma_\rho(x - y + \xi)}{\gamma_\rho(\xi)} = \exp\Par{\frac{1}{2\rho^2}\Par{\inprod{2\xi}{y - x} - \norm{x - y}^2}}. \]
If $\xi$ is large in the direction $y - x$, then the first term in the exponential dominates. Standard Gaussian tail bounds show the measure of such poorly-behaved $\xi$ is small, but we do not have an a priori upper bound on $\nabla^2 f$ at the corresponding points (as $f$ is possibly nonsmooth). Hence, it is unclear how to quantify the effect of these points. We leverage the simple observation that $f_{\rho} = f \ast \gamma_\rho$ is the convolution of $\gamma_{\rho/2}$ and that $f_{\rho/2}$ is a smooth function with a bounded Hessian. Consequently, $\nabla^2 f_\rho(x) \approx \nabla^2 f_\rho(y)$ does hold for $x$ and $y$ in a ball of small radii $r \ll \rho$ up to a small additive factor (which we can control by choosing the radii). To formalize this, we use the following notation for comparing deviations between a pair of PSD matrices in the following \Cref{def:mat_approx}. We then bound the stability of Hessians of the Gaussian convolution in \Cref{lem:conv_stable}.

\begin{definition}[Matrix approximation]\label{def:mat_approx}
	We say that PSD $\ma\in \R^{d \times d}$ is an \emph{$(\eadd, \emul)$-approximation} to
	PSD $\mb \in \R^{d \times d}$ if $
	\ma \preceq \exp(\emul) \mb + \eadd \id_d$ and $\mb \preceq \exp(\emul) \ma + \eadd \id_d$.
\end{definition}

We choose this Definition~\ref{def:mat_approx} because it is symmetric: $\ma$ is an $(\eadd, \emul)$-approximation of $\mb$ if and only if $\mb$ is an $(\eadd, \emul)$-approximation of $\ma$. This symmetry reflects our setting of comparing Hessians of the Gaussian convolution for pairs of points. It is straightforward that Definition~\ref{def:mat_approx} implies other notions of additive-multiplicative approximation, e.g., the less-symmetric alternative $\exp(-\emul) \mb - \eadd \id_d \preceq \ma \preceq \exp(\emul) \mb + \eadd \id_d$.

\begin{lemma}\label{lem:conv_stable}
Let $f: \R^d \to \R$ be convex and $L$-Lipschitz, and $\rho > 0$. Then for $x, y \in \R^d$, $\delta \in (0, 1)$, $\nabla^2 f_\rho(x)$ is an $(\eadd, \emul)$-approximation to $\nabla^2 f_\rho(y)$, following Definition~\ref{def:mat_approx}, for
\[\emul \defeq \frac{\norm{x - y}^2}{\rho^2} + \frac{2\norm{x - y}\sqrt{\log \frac 1 \delta}} \rho,\; \eadd \defeq \frac{\sqrt{2}L\delta}{\rho}.\]
\end{lemma}
\begin{proof}
Without loss of generality (as shifting by a constant vector does not affect the problem assumptions), let $y = \0_d$. Furthermore, let $g \defeq f_{\rho}$ and $h \defeq f_{\rho / \sqrt{2}}$, so that $g = h \ast \gamma_{\rho / \sqrt{2}}$. Note that by Fact~\ref{fact:gconv}, $\nabla^2 h$ is bounded by $\frac{\sqrt 2 L}{\rho} \id_d$ pointwise, and for all $\xi \in \set_r \defeq \{\xi \in \R^d \mid \inprod{x}{\xi} \ge - r^2\}$,
\begin{equation}\label{eq:small_region}
\gamma_{\rho / \sqrt 2}\Par{\xi} 
= \exp\Par{\frac{\norm{x}^2-2\inprod{x}{\xi}}{\rho^2}} \gamma_{\rho / \sqrt 2}(\xi - x) 
\le \exp\Par{\frac{\norm{x}^2 + 2r^2}{\rho^2}}\gamma_{\rho / \sqrt 2}(\xi - x).
\end{equation}
We hence have, by Fact~\ref{fact:gconv} applied to $h$ and $g$,
\begin{align*}
\nabla^2 g(x) &= \int_{\set_r} \nabla^2 h(x - \xi) \gamma_{\rho / \sqrt 2}(\xi) \dd \xi + \int_{\R^d \setminus \set_r} \nabla^2 h(x - \xi) \gamma_{\rho / \sqrt 2}(\xi) \dd \xi \\
&\preceq \exp\Par{\frac{\norm{x}^2 + 2r^2}{\rho^2}} \int_{\set_r} \nabla^2 h(x - \xi) \gamma_{\rho / \sqrt 2}(\xi - x) \dd \xi + \frac{\sqrt{2}L}{\rho} \Par{\Pr_{\xi \sim \Nor(0, \frac{\rho^2} 2 \id_d)}\Brack{\xi \in \set_r}} \id_d \\
&\preceq \exp\Par{\frac{\norm{x}^2 + 2r^2}{\rho^2}} \nabla^2 g(\0_d) + \frac{\sqrt{2}L}{\rho} \exp\Par{-\frac{r^4}{\rho^2\norm{x}^2}}\id_d.
\end{align*}
In the second line, we used \eqref{eq:small_region} and that $\nabla^2 h \preceq \frac{\sqrt{2}L}{\rho} \id_d$, and in the last line, we used standard tail bounds on the Gaussian error function \cite[Eq.\ 7.8.3]{DLMF}. The conclusion follows by substituting the specific value $r = \sqrt{\rho\norm{x}} \cdot \log^{\frac 1 4}(\frac 1 \delta) $, and using symmetry of \Cref{def:mat_approx} in $x$ and $y$.
\end{proof}

Our ball optimization oracle objective $\frlbx$ in \eqref{eq:frldef} is regularized, so when the additive term is dominated by the regularizer's Hessian, we can show $\frlbx$ is multiplicatively second-order stable. This is the key fact that we use to facilitate the implementation of our ball optimization oracles.

\begin{corollary}\label{cor:conv_reg_stable}
Let $\lam > 0$ and suppose that $y \in \ball_x(r)$ for $0 < r \le \frac \rho 6 \cdot \log^{-\half}(\frac{2L}{\lam\rho})$. Then, $\nabla^2 \frlbx(x)$ is a $(0, \log 2)$-approximation to $\nabla^2 \frlbx(y)$, following Definition~\ref{def:mat_approx}.
\end{corollary}
\begin{proof}
Note that $x \in \ball_y(r)$ if and only if $y \in \ball_x(r)$, so by symmetry it suffices to show $\nabla^2 \frlbx(x) \preceq 2\nabla^2 \frlbx(y)$. Also, $\nabla^2 \frlbx(x) = \nabla^2 f_\rho(x) + \lam \id_d$, and by Lemma~\ref{lem:conv_stable} and our parameter settings,
\begin{align*}
\nabla^2 f_\rho(x) - \frac{\sqrt{2}L\delta}{\rho} \id_d \preceq 2\nabla^2 f_\rho(y)
\text{ and }
\lam \id_d + \frac{\sqrt{2}L\delta}{\rho} \id_d \preceq 2\lam \id_d,
\end{align*}
for $\delta \defeq \frac{\lam\rho}{\sqrt{2}L}$.
Adding the above two inequalities yields the conclusion.
\end{proof}

\subsection{Hessian optimization without Hessian approximation} 
\label{sec:intro:hess_opt}

Multiplicative Hessian stability in small balls was a key building block of several algorithms in \cite{CarmonJJJLST20}, which introduced the ball acceleration framework. Specifically, \cite{CarmonJJJLST20} considered problems such as logistic and $\ell_p$ regression, where the objective's Hessian is both explicit and locally multiplicatively stable. This stability allows for efficient ball optimization via variants of Newton's method (e.g., gradient descent preconditioned by the objective's Hessian at the ball center).

In our setting, the use of such Newton's methods for ball optimization is complicated by two factors: our parallel implementation requirement, and the fact that we only have implicit access to $\nabla^2 \frlbx$, as it involves evaluating an integral. One useful characterization is that\footnote{The $\lam \id_d$ component of $\nabla^2 \frlbx$ is explicit, so it suffices to evaluate $\nabla^2 f_\rho$.}
\begin{align*}
\nabla^2 f_\rho(\bx) = \frac 1 {\rho^2} \int_{\R^d} \underbrace{\nabla f(x + \xi)\xi^\top}_{\defeq \mm_{\xi}} \gamma_\rho(\xi) \dd \xi,
\end{align*}
so a natural way to proceed is to estimate $\nabla^2 f_\rho(\bx)$ as the average of a small number of randomly-sampled $\mm_\xi$. Unfortunately, matrix concentration bounds such as the matrix Bernstein inequality (which are tight in the worst case \cite{Tropp15}) yield sample complexities which depend on 
\begin{equation}\label{eq:vdef}\max\Brace{\underbrace{\normop{\E \Brack{\mm_\xi \mm_\xi^\top}}}_{\defeq V_1}, \underbrace{\normop{\E \Brack{\mm_\xi^\top \mm_\xi}}}_{\defeq V_2}},\end{equation}
for measuring convergence of random averages to $\nabla^2 f_\rho(\bx)$. In our case, $V_2$ can be significantly smaller than $V_1$: the former can be upper bounded by $L^2$ (as $\E \xi\xi^\top = \rho^2\id_d$), but the latter grows as $L^2 d$ (as $\E \xi^\top\xi = \rho^2 d$). A similar issue arises if one replaces the use of $\mm_\xi$ with its symmetrized counterpart $\ms_\xi \defeq \half(\mm_\xi + \mm_\xi^\top)$. This results in requiring at least $d$ samples for Hessian approximation which is too many for our purposes (note that, e.g., the entire query complexity of \cite{CarmonJJLLST23} is $o(d)$ for moderate $\eps$), and appears to be an obstacle for use of Newton's method.

We circumvent this obstacle by treating the implementation of each Newton step as a stochastic optimization problem, which breaks the symmetry between the dependence on $V_1$ and $V_2$. Specifically, each Newton iteration requires approximately solving a problem 
\begin{equation}\label{eq:newton_def}\min_{x \in \ball_{\bx}(r)} \inprod{g}{x} + x^\top \nabla^2 f_\rho(\bx)x + \frac \lam 2 \norm{x - \bx}^2,\end{equation}
for some vector $g$. We can therefore design a stochastic estimate $g + 2\mm_\xi x$ of the gradient of the objective in \eqref{eq:newton_def}, whose second moment only depends on $V_2 = \E \mm_\xi^\top\mm_\xi$ and not $V_1$. Interestingly, using $\ms_\xi$ to estimate the gradient of \eqref{eq:newton_def} would run into the same issue as before (where the convergence rate also depends on $V_1 \approx L^2 d$), so using asymmetric estimates is crucial for our analysis.

\subsection{Parallel maintenance of rank-one updates} 
\label{sec:overview:parallel-rank-1}

An additional challenge is implementing our strategy for solving \eqref{eq:newton_def} efficiently in parallel. We show in Section~\ref{ssec:binary_search} how to remove the constraint in \eqref{eq:newton_def} by developing a binary search procedure for an appropriate Lagrange multiplier, so it suffices to optimize over $\R^d$, subject to additional regularization. To facilitate this reduction, in \Cref{ssec:tournament}, we provide a technique for improving our expected error bounds to high probability bounds (similar to a recent technique in \cite{SidfordZ23}).

With these reductions in place it suffices to solve unconstrained variants of \eqref{eq:newton_def} in parallel. In \Cref{ssec:composite}, we provide a general stochastic composite gradient descent algorithm compatible with the stochastic oracles discussed in \Cref{sec:intro:hess_opt}. It then turns out, as intended, that the resulting iterates of this stochastic composite gradient descent algorithm are highly-structured (as each $\mm_\xi$ is rank-one). This structure is captured by the following linear algebraic maintenance problem, which we solve in \Cref{sec:quadratic}, allowing for the parallel implementation of our stochastic gradient method.

\begin{restatable}{problem}{restaterankonemaintainall}\label{prob:rankone_maintain_all}
Let $T \in \N$. For inputs $\{x_0, \{u_t, v_t, w_t\}_{t \in [T]}\} \subset \R^d$ and $\{c_t\}_{t \in [T]} \subset \R$, we wish to compute all $\{x_t\}_{t \in [T]}$ defined by the recurrence relation
\[x_t \defeq c_t\Par{\Par{\id_d - u_t v_t^\top} x_{t - 1}} + w_t.\]
\end{restatable}

In our setting, $c_t$ arises due to the regularization component, $w_t$ captures the first-order part of \eqref{eq:newton_def}, and $u_t, v_t$ capture the (rank-one) estimate of the second-order part of \eqref{eq:newton_def}. Note that if the term $w_t$ did not exist, solving Problem~\ref{prob:rankone_maintain_all} would amount to computing the product of $T$ rank-one matrices in parallel, which can be done using a divide-and-conquer technique (see Lemma~\ref{lem:lowrank}). By using a relatively lightweight combination of divide-and-conquer and fast matrix multiplication, we show that Problem~\ref{prob:rankone_maintain_all} can similarly be solved in polylogarithmic depth and $O(dT^{\omega - 1})$ work.

Applying this parallel implementation of our stochastic composite gradient descent algorithm, though our ``in expectation-to-high probability'' and binary search reductions, yields our ball optimization oracle implementation. When applied in the ball optimization framework to the Gaussian convolution, this then yields our main result, Theorem~\ref{thm:main_formal}. While there are a few steps of indirection, we believe that reducing parallel optimization to stochastic quadratic optimization is an interesting and key contribution by itself. We hope the structural facts that enable this reduction and the algorithmic techniques that make it yield an efficient algorithm may have broader utility. %
\section{Parallel optimization of quadratic subproblems}
\label{sec:quadratic}

In this section, we develop an efficient parallel optimization method for solving structured unconstrained quadratics of the following form: %
\begin{equation}\label{eq:unconstrained}\min_{x \in \R^d}\inprod{g + v}{x} + \norm{x - z}^2_{\mh} + \frac \Lam 2\norm{x}^2,\text{ for } g, v \in \R^d,\; \mh \in \R^{d \times d},\;\Lam \in \R_{\ge 0}.\end{equation}
In particular, we consider a stochastic setting where instead of explicit access to $g$ or $\mh$ we assume sample access to random variables $\tg \in \R^d$ and $\tmh \in \R^{d \times d}$ (not necessarily symmetric), such that
\begin{equation}\label{eq:tgmhdef} \E \tg = g\text{ and }\E \tmh = \mh.\end{equation}
Our consideration of this setting is motivated by the special case when
\begin{equation}\label{eq:specific_gaussian}
g = \nabla f_\rho(z)\text{ and }\mh = \nabla^2 f_\rho(\0_d).
\end{equation}
Objectives of the form \eqref{eq:unconstrained}, \eqref{eq:specific_gaussian} arise in Newton's method for implementing a ball optimization oracle for $\frlbx$ as defined in \eqref{eq:frldef}, where $\bx \gets \0_d$ without loss of generality by shifting the problem domain. The additional quadratic $\inprod{v}{x} + \Lam \norm{x}^2$ arises due to regularization and a binary search on a Lagrange multiplier to enforce the domain constraint in \eqref{eq:frldef}. The two parts of this reduction (Newton's method and binary search) are respectively derived in Sections~\ref{ssec:newton} and~\ref{ssec:binary_search}. 

In Section~\ref{ssec:composite}, we give an initial solver for the problem \eqref{eq:unconstrained} that has an expected error guarantee and we show how to implement this solver in parallel in Section~\ref{ssec:maintenance}. We then show how to boost this guarantee to hold with high probability using a reduction we develop in Section~\ref{ssec:tournament}. Finally, we assemble these components to give the main exports of this section in Section~\ref{ssec:combine}.

\subsection{Composite stochastic optimization}\label{ssec:composite}

In this section, we fix $\Lam \in \R_{\ge 0}$ and $v, z \in \R^d$ throughout, and decompose \eqref{eq:unconstrained} into two parts:
\begin{equation}\label{eq:hdef}
\begin{aligned}
	 h(x) \defeq h_1(x) + h_2(x)
	 \text{ where }
h_1(x) \defeq \inprod{g}{x} + \norm{x - z}^2_{\mh}
	\text{ and }
h_2(x) \defeq \inprod{v}{x} + \frac \Lam 2 \norm{x}^2\;.
\end{aligned}
\end{equation}
We treat the objective in \eqref{eq:unconstrained} as a composite objective $h_1 + h_2$, where we can exactly optimize over $h_2$, and we have stochastic access to $h_1$. Specifically, we use that $g_1(x) \defeq \tg + 2\tmh(x - z)$ is an unbiased estimate of $\nabla h_1(x)$. More broadly, we design an algorithm for minimizing $h_1 + h_2$ under the assumption that for some $L_1, L_2 \in \R_{\ge 0}$, 
\begin{equation}\label{eq:dist_var_bound}\E\Brack{\norm{g_1(x)}^2} \le L_1^2 + L_2^2 \norm{x - z}^2 \text{ for all } x \in \R^d.\end{equation}
To motivate \eqref{eq:dist_var_bound}, Fact~\ref{fact:gconv} shows that in the setting of Problem~\ref{prob:sco}, when $g, \mh$ are as in \eqref{eq:specific_gaussian}, we can use 
\begin{equation}\label{eq:grad_unbiased}
g_1(x) = g(z - \xi) + \frac{2}{\rho^2} \inprod{\xi}{x - z} g(\xi) \text{ for } \xi \sim \Nor(\0_d, \rho^2 \id_d)
\end{equation}
as our unbiased estimator. We give a second moment bound on the estimator in \eqref{eq:grad_unbiased}.

\begin{lemma}\label{lem:second_moment}
In the setting of Problem~\ref{prob:sco}, for $x, z \in \R^d$, where $\E$ is taken over $\xi \sim \Nor(\0_d, \rho^2 \id_d)$ and the randomness of querying $g$ at $z - \xi$ and $\xi$,
\[\E\Brack{\norm{g(z - \xi) + \frac{2}{\rho^2} \inprod{\xi}{x - z}g(\xi)}^2} \le 2L^2 + \frac{8L^2}{\rho^2}\norm{x - z}^2.\]
\end{lemma}
\begin{proof}
Let $a \defeq g(z - \xi)$ and $b \defeq \frac{2}{\rho^2} \inprod{\xi}{x - z}g(\xi)$, and note that $\E\norm{a + b}^2 \le 2\E \norm{a}^2 + 2\E \norm{b}^2$. Further, by definition $\E \norm{a}^2 \le L^2$, so it suffices to bound $\E \norm{b}^2$. We conclude by computing:
\begin{align*}
\E\Brack{\norm{\inprod{\xi}{x - z}g(\xi)}^2} &= \E_{\xi \sim \Nor(\0_d, \rho^2 \id_d)}\Brack{\E \Brack{\norm{g(\xi)}^2 \mid \xi}\inprod{\xi}{x - z}^2} \\
&\le L^2 \E_{\xi \sim \Nor(\0_d, \rho^2 \id_d)}\Brack{\inprod{\xi}{x - z}^2} = L^2\rho^2\norm{x - z}^2.
\end{align*}
\end{proof}
In other words, in the setting of Problem~\ref{prob:sco}, the assumption \eqref{eq:dist_var_bound} holds with $L_1^2 = 2L^2$ and $L_2^2 = \frac{8L^2}{\rho^2}$. We now move to the abstraction of \eqref{eq:hdef}, \eqref{eq:dist_var_bound}, and design a general-purpose algorithm for this optimization problem. At the end of the section, we specialize our method to Problem~\ref{prob:sco}. 

Due to the unconstrained nature of our problem and the dependence \eqref{eq:dist_var_bound} on movement from $z$, we take care to ensure that the iterates of our algorithm do not drift by too much. This is the primary challenge faced in this setting. We give an algorithm (Algorithm~\ref{alg:unconstrained_sgd_gen}) and analysis based on \cite{LacosteJSB12}, using a ``warm-started'' step size schedule to ensure sufficient expected norm bounds on iterates.

\begin{algorithm2e}
	\caption{$\UCS(h_2, z, g_1)$}
	\label{alg:unconstrained_sgd_gen}
	\DontPrintSemicolon
		\codeInput $h_2: \R^d \to \R$, a $\Lambda$-strongly convex function, $z \in \R^d$, and $g_1: \R^d \to \R^d$, an unbiased estimator for $\nabla h_1$ where $h_1: \R^d \to \R$ is convex, and for all $x \in \R^d$, \eqref{eq:dist_var_bound} holds.
  \;
  $x_0 \gets \argmin_{x \in \R^d} h_2(x)$\;
  $T_0 \gets \frac{8L_2^2}{\Lam^2}$\;
  \For{$0 \le t < T$}{
  $\eta_t \gets \frac{2}{\Lambda(t + T_0)}$\;
  $x_{t + 1} \gets \arg\min_{x \in \R^d}\{\eta_t\inprod{g_1(x_t)}{x} + \eta_t h_2(x) + \half\norm{x - x_t}^2\}$\;
  }
  \codeReturn $\xavg \defeq \frac 1 T \sum_{0 \le t < T} x_t$
\end{algorithm2e}

\begin{lemma}\label{lem:gen_sgd_analysis}
Let $h_1: \R^d \to \R$, $h_2: \R^d \to \R$, $g_1: \R^d \to \R^d$, and $z \in \R^d$ satisfy the assumptions of Algorithm~\ref{alg:unconstrained_sgd_gen} for $L_2 \ge \Lam$, and let $x^\star$ minimize $h \defeq h_1 + h_2$. Following notation of Algorithm~\ref{alg:unconstrained_sgd_gen},
\begin{align*}\E\Brack{h(\xavg) } - h(x^\star) &\le \Par{\frac{2L_2^2}{\Lam T} + \frac{\Lam}{4T}}\norm{x_0 - x^\star}^2 
+ \frac{\log(T + T_0)}{\Lam T}\Par{L_1^2 + 2L_2^2\norm{z - x^\star}^2}.\end{align*}
\end{lemma}
\begin{proof}
Throughout the proof, for all $0 \le t < T$ let
\[D_t \defeq \E \Brack{\half \norm{x_t - x^\star}^2} \text{ and } \Phi_t \defeq \E\Brack{h(x_t) } - h(x^\star).\]
By first-order optimality of $x_{t + 1}$,
\[\inprod{g_1(x_t)}{x_{t + 1} - x^\star} + \inprod{\nabla h_2(x_{t + 1})}{x_{t + 1} - x^\star} \le \frac 1 {2\eta_t}\Par{\norm{x_t - x^\star}^2 - \norm{x_{t + 1} - x^\star}^2 - \norm{x_t - x_{t + 1}}^2}.\]
By rearranging and taking an expectation (conditioned on the realization of $x_t$),
\begin{align*}
h_1(x_t) + h_2(x_{t + 1}) - h(x^\star)  + \frac \Lam 2 \norm{x_{t + 1} - x^\star}^2 &\le \inprod{\nabla h_1(x_t)}{x_t - x^\star} + \inprod{\nabla h_2(x_{t + 1})}{x_{t + 1} - x^\star} \\
&\le \frac 1 {2\eta_t}\Par{\norm{x_t - x^\star}^2 - \E\norm{x_{t + 1} - x^\star}^2} \\
&+  \E\Brack{\inprod{g_1(x_t)}{x_t - x_{t + 1}} + \frac{1}{2\eta_t}\norm{x_t - x_{t + 1}}^2} \\
&\le \frac 1 {2\eta_t}\Par{\norm{x_t - x^\star}^2 - \E\norm{x_{t + 1} - x^\star}^2} \\
&+ \frac{\eta_t}{2}\E\Brack{\norm{g_1(x_t)}^2}.
\end{align*}
In the first inequality, we used convexity of $h_1$ and strong convexity of $h_2$, and in the last line, we applied the Cauchy-Schwarz and Young's inequalities to bound the quantity in the third line. Hence, iterating expectations, rearranging, and using the assumption \eqref{eq:dist_var_bound},
\begin{align*}\Phi_t + \Lam D_{t + 1} &\le \E\Brack{h_2(x_t) - h_2(x_{t + 1})} + \frac 1 {\eta_t}\Par{D_t - D_{t + 1}} + \frac{\eta_t}{2}\Par{L_1^2 + L_2^2\E\Brack{\norm{x_t - z}^2}} \\
&\le \E\Brack{h_2(x_t) - h_2(x_{t + 1})} + \frac 1 {\eta_t}\Par{D_t - D_{t + 1}} + \frac{\eta_t}{2}\Par{L_1^2 + 2L_2^2\norm{z - x^\star}^2 + 4L_2^2 D_t}. 
\end{align*}
Moreover, note that for the given choice of parameters, i.e., using $T_0 = \frac{8L_2^2}{\Lam^2}$,
\begin{align*}
\frac 1 {\eta_t} + 2\eta_t L_2^2 &\le \frac{\Lam(t + T_0)}{2} + \frac{4L_2^2}{\Lam T_0} = \frac{\Lam(t + T_0 + 1)}{2}, \\
\frac 1 {\eta_t} + \Lam &= \frac{\Lam(t + T_0)}{2} + \Lam = \frac{\Lam (t + T_0 + 2)}{2}.
\end{align*}
Combining the above two displays, we have
\begin{align*}\Phi_t &\le \E\Brack{h_2(x_t) - h_2(x_{t + 1})} + \frac{\Lam(t + T_0 + 1)}{2} D_t - \frac{\Lam(t + T_0 + 2)}{2} D_{t + 1} \\
&+ \frac{\eta_t}{2}\Par{L_1^2 + 2L_2^2\norm{z - x^\star}^2}.\end{align*}
Therefore, by telescoping over $T$ iterations, and using minimality of $x_0$ with respect to $h_2$,
\[\frac 1 T \sum_{0 \le t < T} \Phi_t \le \frac{\Lam (T_0 + 1)}{4T} \norm{x_0 - x^\star}^2 + \Par{\frac 1 T \sum_{0 \le t < T} \frac{\eta_t} 2}\Par{L_1^2 + 2L_2^2\norm{z - x^\star}^2}. \]
Applying convexity of $h$ and plugging in our parameter choices yields the claim, where we bound the partial harmonic sequence $\sum_{0 \le t < T} \frac 1 {t + T_0} \le \sum_{t = 8}^{T + T_0} \frac 1 t \le \log(t + T_0)$.
\end{proof}

We now state Algorithm~\ref{alg:unconstrained_sgd}, our specialization of Algorithm~\ref{alg:unconstrained_sgd_gen} to the problem \eqref{eq:unconstrained}, \eqref{eq:specific_gaussian}. In Line~\ref{line:sgd_update} of Algorithm~\ref{alg:unconstrained_sgd}, we used the definition of $h_2$ from \eqref{eq:hdef}. We observe that the update in Line~\ref{line:sgd_update} can be conveniently written in closed form as
\begin{equation}\label{eq:sgd_explicit}
	x_{t + 1} \gets \frac{1}{1 + \eta_t\Lam}\Par{x_t - \eta_t v - \eta_t g'_t - \frac{2\eta_t}{\rho^2}\inprod{\xi_t}{x_t - z}g_t}.
\end{equation}
We demonstrate in Section~\ref{ssec:maintenance} how to support efficient parallel maintenance of weighted averages of iterates undergoing updates of the form \eqref{eq:sgd_explicit}. For now, we give an error bound following Lemma~\ref{lem:gen_sgd_analysis}.

\begin{algorithm2e}
	\caption{$\UCSS(\Lam, \rho, z, v, T, f)$}
	\label{alg:unconstrained_sgd}
	\DontPrintSemicolon
		\codeInput $\Lam, \rho \ge 0$, $z, v \in \R^d$, $T \in \N$, $f$ in the setting of Problem~\ref{prob:sco} \;
  $x_0 \gets -\frac{1}{2\Lam}v$\;
  $T_0 \gets \frac{64L^2}{\rho^2\Lam^2}$\;
  \For{$0 \le t < T$}{
  $\xi_t \sim \Nor(\0_d, \rho^2 \id_d)$\;
  $g_t \gets g(\xi_t)$, $g'_t \gets g(z - \xi_t)$\;
  }
		\For{$0 \le t < T$}{
            $\eta_t \gets \frac{2}{\Lam(t  + T_0)}$\;
		$x_{t + 1} \gets \arg\min_{x \in \R^d}\{\eta_t\inprods{ g'_t + \frac 2 {\rho^2}  \inprod{\xi_t}{x_t - z} g_t }{x} + \eta_t h_2(x) + \half \norm{x - x_t}^2\}$ \;\label{line:sgd_update}
		}
    \codeReturn $\xavg \defeq \frac 1 T \sum_{0 \le t < T} x_t$
\end{algorithm2e}

\begin{corollary}\label{cor:specific_sgd_analysis}
Following the notation in \eqref{eq:hdef} and Algorithm~\ref{alg:unconstrained_sgd}, let $\xslzv$ minimize \eqref{eq:unconstrained} under the setting \eqref{eq:specific_gaussian}, and suppose $\max(\norm{z}, \norm{x_0}) \le \rho$ and $\rho \le \frac L \Lam$. Then,
\begin{align*}\E h(\xavg) - h(\xslzv) &\le \Par{\frac{66L^2\log(T + T_0)}{\Lam T} + \frac{\Lam \rho^2}{2T}}\Par{1 + \frac{ \norm{\xslzv}^2}{\rho^2}} .\end{align*}
\end{corollary}
\begin{proof}
In light of Lemma~\ref{lem:second_moment} and the fact that the gradient estimator in \eqref{eq:grad_unbiased} is unbiased for $\nabla h_1$ defined in \eqref{eq:hdef}, we can apply Lemma~\ref{lem:gen_sgd_analysis} with $h_1$, $h_2$ as in \eqref{eq:hdef}, and
$L_1^2 \defeq 2L^2$ and $L_2 \defeq 8L^2 / \rho^2$. The conclusion follows from Lemma~\ref{lem:second_moment} once we simplify using $\max(\norm{x_0}, \norm{z}) \le \rho$.
\end{proof}

To gain some intuition for Corollary~\ref{cor:specific_sgd_analysis}, note that it shows if $\norms{\xslzv} \ll \rho$ and our target error is $\Theta(\Lam r^2)$, we recover the standard $\frac{L^2}{\Lam T}$ rate of strongly convex optimization under a bounded-variance gradient oracle, up to a log factor. We show how to combine this guarantee with a binary search in Section~\ref{sec:framework} to efficiently solve constrained optimization problems, as required by Proposition~\ref{prop:ball_accel}.

\subsection{Parallel maintenance of rank-$1$ updates}\label{ssec:maintenance}

In this section, we give our parallel solution to Problem~\ref{prob:rankone_maintain_all}, reproduced here for convenience.

\restaterankonemaintainall*

Our updates in Algorithm~\ref{alg:unconstrained_sgd}, as stated in \eqref{eq:sgd_explicit}, are exactly of the form in Problem~\ref{prob:rankone_maintain_all}, with 
\begin{equation}\label{eq:problem_example}
\begin{aligned}
c_t \gets \frac 1 {1 + \eta_{t - 1}\Lam},\; u_t \gets \frac{2\eta_{t - 1}}{\rho^2} g_{t - 1},\; v_t \gets \xi_{t - 1},\\
w_t \gets -c_t\Par{\eta_{t - 1}(v + g'_{t - 1}) + \frac{2\eta_{t - 1}}{\rho^2}\inprod{\xi_{t - 1}}{z}g_{t - 1}}. 
\end{aligned}
\end{equation}
Moreover, all of the $\{c_t, u_t, v_t, w_t\}_{t \in [T]}$ can be queried and precomputed with $\tO(1)$ depth and $O(dT)$ work.
Accordingly, it suffices to solve Problem~\ref{prob:rankone_maintain_all} to give a parallel implementation.
As a warmup to our overall solution, we first give our parallel solution to the following simpler Problem~\ref{prob:rankone_maintain}.

\begin{problem}\label{prob:rankone_maintain}
In the setting of Problem~\ref{prob:rankone_maintain_all}, we wish to compute $x_T$.
\end{problem}

We then modify our strategy to solve Problem~\ref{prob:rankone_maintain_all}, at a slightly larger parallel depth. Our solution to Problem~\ref{prob:rankone_maintain} follows straightforwardly from maintaining matrix products in a dyadic fashion, using the following observation (Lemma~\ref{lem:lowrank}) on maintaining low-rank updates of the identity. We combine this with a variant of a parallel prefix sum maintenance strategy for recursive matrix-vector products.

\begin{lemma}\label{lem:lowrank}
Let $\ma_0, \mb_0, \ma_1, \mb_1 \in \R^{d \times r}$ for $r \in [d]$. In depth  $O(\log d)$ and work $O(dr^{\omega - 1})$, we can compute $\ma, \mb \in \R^{d \times 2r}$ such that $\id_d - \ma\mb^\top = (\id_d - \ma_0 \mb_0^\top)(\id_d - \ma_1\mb_1^\top)$.
\end{lemma}
\begin{proof}
Note that it suffices to choose
\begin{align*}
\ma = \begin{pmatrix}
 \ma_0 & \ma_1 - \ma_0\mb_0^\top\ma_1
\end{pmatrix}\text{ and } \;
\mb = \begin{pmatrix}
\mb_0 & \mb_1
\end{pmatrix}.
\end{align*}
The bottleneck is evaluating $\ma_0 \mb_0^\top \ma_1$ which takes work $O(dr^{\omega - 1})$ and depth $O(\log d)$.
\end{proof}

Building upon Lemma~\ref{lem:lowrank}, we next give a solution to Problem~\ref{prob:rankone_maintain} when all $c_t = 1$.

\begin{lemma}\label{lem:noC_maintain}
If $c_t = 1$ for all $t \in [T]$, we can solve Problem~\ref{prob:rankone_maintain} with depth $O(\log(d)\log (T))$ and work $O(dT^{\omega - 1})$.
\end{lemma}
\begin{proof}
Throughout this proof let $\ell \defeq \lfloor \log_2 T\rfloor + 1$, define $w_0 \defeq x_0$, and for $t > T$ let $u_t = v_t = w_t \defeq \0_d$. We also define for $s \le t$, where all matrix products are evaluated right-to-left,
\[\mm_{t:s} \defeq \prod_{r = s}^{t} \Par{\id_d - u_r v_r^\top},\]
so that $\mm_{T:1} = (\id_d - u_Tv_T^\top)\ldots(\id_d - u_1v_1^\top)$. As all iterates after $T$ do not change, we observe that
\begin{equation}\label{eq:xtdef}x_T = x_{2^\ell} = \sum_{t = 0}^{2^\ell - 1} \mm_{2^\ell:t+1} w_t,\end{equation}
since by definition, $w_{2^\ell} = \0_d$.
For notational convenience, we define
\begin{equation}\label{eq:mmijdef}\mm_{(i, j)} \defeq \mm_{2^i j:2^i(j-1) + 1}\end{equation}
for each $0 \le i \le \ell$ and $j \in [2^{\ell - i}]$. We observe that with $O(dT^{\omega - 1})$ work, we can explicitly compute $\{\ma_{(i, j)}, \mb_{(i, j)}\}_{0 \le i \le \ell, j \in [2^{\ell - i}]}$ such that
\begin{align*}
\mm_{(i, j)} = \id_d - \ma_{(i,j)}\mb_{(i, j)}^\top,
\text{ for } \ma_{(i, j)}, \mb_{(i, j)} \in \R^{d \times 2^i}.
\end{align*}
To see this, we clearly can compute all $\{\ma_{(0,j)},\mb_{(0, j)}\}_{j \in [2^\ell]}$ with $O(dT)$ work and $O(1)$ depth. Further, for $0 \le i < \ell$, assuming access to all $\{\ma_{(i,j)},\mb_{(i, j)}\}_{j \in [2^{\ell - i}]}$, we can apply Lemma~\ref{lem:lowrank} in parallel to compute each required $\ma_{(i + 1, j)}, \mb_{(i + 1, j)}$ with work $O(d(2^i)^{\omega - 1})$, incurring a total work of
\[2^{\ell - i} \cdot O\Par{d(2^i)^{\omega - 1}} = O(dT) \cdot (2^i)^{\omega - 2}.\]
Summing over all $0 \le i \le \ell$ yields a geometric sequence with dominant term $O(dT^{\omega - 1})$ as desired. This procedure can be implemented in depth $O(\log(T)\log(d))$ by repeatedly applying Lemma~\ref{lem:lowrank}.
Next, for each $0 \le i \le \ell$ and $j \in [2^{\ell - i}]$, define
\begin{equation}\label{eq:xijdef}
x_{(i, j)} \defeq \sum_{k \in [2^i]} \mm_{2^i j: 2^i (j - 1) + k} w_{2^i(j - 1) - 1 + k},
\end{equation}
such that by inspection, the following recursion holds for $i \in [\ell]$:
\begin{equation}\label{eq:xrecurse}
x_{(i, j)} = \mm_{2^i j : 2^i j - 2^{i - 1} + 1} x_{(i-1,2j-1)} + x_{(i-1, 2j)}.
\end{equation}
For example,
\begin{align*}
x_{(3, 1)} &= \mm_{8:1} w_0 + \mm_{8:2} w_1 + \mm_{8:3} w_2 + \mm_{8:4} w_3 + \mm_{8:5} w_4 + \mm_{8:6} w_5 + \mm_{8:7} w_6 + \mm_{8:8} w_7 \\
&= \mm_{8:5}\underbrace{\Par{\mm_{4:1} w_0 + \mm_{4:2} w_1 + \mm_{4:3} w_2 + \mm_{4:4} w_3 }}_{x_{(2, 1)}} + \underbrace{\mm_{8:5} w_4 + \mm_{8:6} w_5 + \mm_{8:7} w_6 + \mm_{8:8} w_7}_{x_{(2, 2)}}.
\end{align*}
Our goal is simply to compute $x_{(\ell, 1)} = x_T$, where we recall \eqref{eq:xtdef}. First, we clearly can compute all $x_{(0, j)}$ for $j \in [2^\ell]$ with $O(dT)$ work. Further, for $0 \le i < \ell$, assuming access to all $x_{(i, j)}$ for $j \in [2^{\ell - i}]$ and all $\{\ma_{(i, j)}, \mb_{(i,j)}\}_{j \in [2^{\ell - i}]}$, we claim we can compute all $x_{(i + 1, j)}$ for $j \in [2^{\ell + 1 - i}]$ in parallel incurring a total work of $O(dT)$. To see this, to compute each $x_{(i + 1, j)}$ via the recursion \eqref{eq:xrecurse}, we require one vector addition, and one matrix-vector multiplication through 
\[\mm_{2^i j:2^i j - 2^{i - 1} + 1} = \id_d - \ma_{(i - 1, 2j)}\mb_{(i - 1, 2j)}^\top\] which can be performed with $O(d2^i)$ work. Therefore, the total work required to compute all $x_{(i + 1, j)}$ is bounded by $O(d2^\ell) = O(dT)$. Finally, summing over all $0 \le i \le \ell$ the total work of these computations is bounded by $O(dT\log T)$ which does not asymptotically dominate the work. The depth of this recursive computation is again bounded by $O(\log(T)\log(d))$.
\end{proof}

We conclude with a simple extension of Lemma~\ref{lem:noC_maintain} to general $\{c_t\}_{t \in [T]}$, giving our full solution.

\begin{corollary}\label{cor:prob2solve}
We can solve Problem~\ref{prob:rankone_maintain} with depth $O(\log(T)\log(d))$ and work $O(dT^{\omega - 1})$.
\end{corollary}
\begin{proof}
First, we may assume all $\{c_t\}_{t \in [T]}$ are nonzero, else we can begin the recursion in Problem~\ref{prob:rankone_maintain} starting from the index right after the last zero value. 
Under this assumption, by writing $C_t \defeq \prod_{s\in[t]}c_s$ and $x_t = C_t y_t$ for all $t \ge 0$, we have the equivalent recurrence 
\begin{align*}
y_t = C_t^{-1}\Par{c_t\Par{\id_d - u_tv_t^\top}C_{t - 1}y_{t - 1} + w_t} = \Par{\id_d - u_tv_t^\top}y_{t - 1} + C_t^{-1}w_t.
\end{align*}
Further, computing all $\{C_t\}_{t \in [T]}$ in the same dyadic fashion used to compute the $\mm_{(i, j)}$ in Lemma~\ref{lem:noC_maintain} can be performed in $O(\log T)$ depth and $O(T \log T)$ work.
Hence it suffices to apply Lemma~\ref{lem:noC_maintain} to an instance of Problem~\ref{prob:rankone_maintain} with all $c_t = 1$ and inputs 
\[\{x_0, \{u_t, v_t, C_t^{-1}w_t\}_{t \in [T]}\}\] and multiply the final output vector (corresponding to $y_T$) by $C_T$. The scalings $\{C_t^{-1}w_t\}_{t \in [T]}$ can be performed using $O(1)$ depth and $O(dT)$ work by computing each in parallel.
\end{proof}

We now show how modifying the strategy for Problem~\ref{prob:rankone_maintain} also yields an efficient parallel solution for the generalization in Problem~\ref{prob:rankone_maintain_all}. As before, we first handle the case where all $c_t = 1$. 

\begin{lemma}\label{lem:noC_maintain_all}
If $c_t = 1$ for all $t \in [T]$, we can solve Problem~\ref{prob:rankone_maintain_all} with depth $O(\log^2 (T)\log(d))$ and work $O(dT^{\omega - 1})$.
\end{lemma}
\begin{proof}
We follow notation of Lemma~\ref{lem:noC_maintain}, and assume that in depth $O(\log (T)\log(d))$ and work $O(dT^{\omega - 1})$, we have precomputed all $\mm_{(i, j)}$ and $x_{(i, j)}$ for $0 \le i \le \ell$ and $j \in [2^{\ell - i}]$.
Define $z_s \defeq x_s - w_s$ for all $s \in [2^\ell]$, and let $\mathcal{T}(\ell)$ be the total work it takes to compute all $\{z_s\}_{s \in [2^\ell]}$ in an instance of Problem~\ref{prob:rankone_maintain_all}, given access to all $\mm_{(i, j)}$ and $x_{(i, j)}$ defined in \eqref{eq:mmijdef} and \eqref{eq:xijdef}. In particular, \eqref{eq:xtdef} holds with the left-hand side replaced with $z_s$ and the right-hand side's summation ending at $s - 1$. We claim
\begin{equation}\label{eq:trecurse}
\mathcal{T}(\ell) = 2\mathcal{T}(\ell - 1) + O(dT\log T) \implies \mathcal{T}(\ell) = O(dT\log^2 T),
\end{equation}
which gives the total work bound because adding $w_s$ to each $z_s$ can be done in constant depth and $O(dT)$ work which does not dominate. Clearly, computing all $\{z_s\}_{s \in [2^{\ell - 1}]}$ can be done within work $\mathcal{T}(\ell - 1)$. Moreover, for each $2^{\ell - 1} < s \le 2^\ell$, note that
\begin{align*}
z_s = \underbrace{\sum_{t = 0}^{2^{\ell - 1} - 1}\mm_{s:t + 1} w_t}_{\defeq u_s} + \underbrace{\sum_{t = 2^{\ell - 1}}^{s - 1} \mm_{s:t + 1} w_t}_{\defeq v_s}.
\end{align*}
Computing all $\{v_s\}_{2^{\ell - 1}<s\le 2^\ell}$ can be done within work $\mathcal{T}(\ell - 1)$, as these constitute an independent copy of the problem over $2^{\ell - 1}$ iterations. Finally, we complete the proof of \eqref{eq:trecurse} by showing we can compute all $\{u_s\}_{2^{\ell - 1}<s\le 2^\ell}$ using $O(dT\log T)$ work and $O(\log^2(T)\log(d))$ depth. Define $u^\star \defeq \sum_{t = 0}^{2^{\ell - 1} - 1} \mm_{2^{\ell - 1}:t+1} w_t$, and note that
\[u_s = \mm_{s:2^{\ell - 1} + 1}u^\star\text{ for all } 2^{\ell - 1}<s\le 2^\ell.\]
Since we have access to all the $\mm_{(i, j)}$, we can compute the $u_s$ in a dyadic fashion, i.e., we first compute $u_{2^{\ell - 1} + 2^{\ell - 2}}$ and $u_{2^\ell}$ using a single matrix multiplication each, and then $u_{2^{\ell - 1} + 2^{\ell - 3}}$ and $u_{2^{\ell - 1} +3\cdot 2^{\ell - 3}}$, and so on. The work cost of multiplying by a matrix $\mm_{(i, j)}$ is $O(d2^i)$, so the overall work of computing all $\{u_s\}_{2^{\ell - 1}<s\le 2^\ell}$ is then indeed bounded by
\begin{align*}
O(d2^{\ell - 1}) + O(2 \cdot d2^{\ell - 2}) + O(2^2 \cdot d2^{\ell - 3}) + \ldots = O(dT\log T),
\end{align*}
as claimed. To see the depth bound, we can solve the two instances of $\mathcal{T}(\ell - 1)$ independently, leading to a recursion depth of $O(\log T)$. The sequential depth of each recursion layer (due to computing the $\{u_s\}_{2^{\ell - 1}<s\le 2^\ell}$) is bottlenecked by $O(\log (T)\log(d))$ due to the use of $O(\log T)$ matrix multiplications, each of which takes depth $O(\log d)$. Thus, overall the depth is $O(\log^2 (T)\log(d))$.
\end{proof}

By using the same reduction as in Corollary~\ref{cor:prob2solve}, we extend our solution in Lemma~\ref{lem:noC_maintain_all} to the general setting of Problem~\ref{prob:rankone_maintain_all}, giving our main result.

\begin{proposition}\label{prop:prob3solve}
We can solve Problem~\ref{prob:rankone_maintain_all} with depth $O(\log^2 (T)\log(d))$ and work $O(dT^{\omega - 1})$.
\end{proposition}
\begin{proof}
We first consider the case where all $\{c_t\}_{t \in [T]}$ are nonzero. Define the sequence $\{y_t\}_{t \in [T]}$ as in Corollary~\ref{cor:prob2solve}, which can be computed within the depth and work budgets given by Lemma~\ref{lem:noC_maintain_all}. Since
\[x_t = C_t y_t \text{ for all } t \in [T],\]
it suffices to compute all $\{C_t\}_{t \in [T]}$ and perform the scalings $C_t y_t$ in parallel, which can be done within the stated budgets by the arguments of Corollary~\ref{cor:prob2solve}. Finally, in the case where some $c_t = 0$, we can split the problem into independent contiguous blocks of nonzero $c_t$ values whose total sizes sum to at most $T$ and which can be solved in parallel. Since the claimed work is superlinear in $T$, it remains correct after operating on each contiguous block separately.
\end{proof}

As a consequence of Proposition~\ref{prop:prob3solve}, we have the following complexity bounds on Algorithm~\ref{alg:unconstrained_sgd}.

\begin{corollary}\label{cor:impl_sgd}
Following the notation in \eqref{eq:hdef} and Algorithm~\ref{alg:unconstrained_sgd}, let $S \in [T]$ be arbitrary. We can implement $T$ iterations of Algorithm~\ref{alg:unconstrained_sgd} using 
\begin{gather*}
O\Par{\dQuery + \frac T S \cdot \log^2(S)\log(d)} \text{ depth, and }
O\Par{T \cdot \tQuery + dTS^{\omega - 2}} \text{ work.}
\end{gather*}
\end{corollary}
\begin{proof}
Assume for simplicity that $S$ divides $T$, else we can obtain the result by increasing $T$ by a constant factor.
Recall from \eqref{eq:sgd_explicit}, \eqref{eq:problem_example} that implementing Algorithm~\ref{alg:unconstrained_sgd} is an instance of Problem~\ref{prob:rankone_maintain_all}, where we are required to compute the average iterate. Moreover we can compute all the inputs to Problem~\ref{prob:rankone_maintain_all} in parallel, which gives the query depth and query complexity. Our strategy is to use Proposition~\ref{prop:prob3solve} for every $S$ consecutive iterations, which gives us all the iterates in the stated computational depth and complexity by applying Proposition~\ref{prop:prob3solve}, $\frac T S$ times. Given all the iterates we can clearly output the average within the stated computational depth and complexity.
\end{proof}

\subsection{High-probability error bound reduction}\label{ssec:tournament}

We now give a simple reduction from an algorithm which returns an approximate minimizer with high probability, to one which returns an expected approximate minimizer. Our reduction assumes access to a bounded-variance gradient estimator. We note that a similar procedure appears as Section 4.2 of \cite{SidfordZ23}, but it does not quite suffice for our purposes due to the composite nature of our objective. We provide a different proof for completeness, which also shows a slightly stronger fact that the approximately-optimal point returned comes from the original set of candidates.

Finally, we note that our reduction has implications on the query complexity of high-probability stochastic convex optimization (i.e., Problem~\ref{prob:sco}) in the non-parallel setting. In particular, it shows that the expected error guarantee in Problem~\ref{prob:sco} can be boosted to have failure probability $\le \delta$ at a $\textup{polylog}(\frac 1 \delta)$ overhead in the query complexity. Such a result is classical when $g(x)$ satisfies stronger tail bounds (such as a sub-Gaussian norm), but to our knowledge the corresponding result in the bounded variance setting (as in Problem~\ref{prob:sco}) was only obtained recently by \cite{CarmonH24}. We do note that \cite{CarmonH24}'s approach yields an improved polylogarithmic overhead in $\frac 1 \delta$, which they show is optimal; we find it interesting to explore if different tradeoffs in Proposition~\ref{prop:high_prob_reduction} yield the same optimal result.

\begin{proposition}\label{prop:high_prob_reduction}
Let $h: \R^d \to \R$ be differentiable with minimizer $x^\star$, and assume $h(x) = h_1(x) + h_2(x)$ for all $x \in \R^d$ and we can evaluate $h_2(x)$ for any $x \in \R^d$. Further, suppose $S \defeq \{x_i\}_{i \in [k]}$ has $\norm{x_i - x_j} \le R$ for all $i, j \in [k]$, and $\min_{i\in[k]}h(x_i) - h(x^\star) \le \eps$. Finally, suppose $g_1(x)$  is an unbiased estimate for $\nabla h_1(x)$ and $\E[\norm{g_1(x)}^2] \le L^2$ for all $x$ in the convex hull of $S$. For $\delta \in (0, 1)$, there is an algorithm which returns $x \in S$ with $h(x) - h(x^\star) \le 2\eps$ with probability $\ge 1 - \delta$, using 
\begin{align*}
O\Par{\frac{L^2 R^2}{\eps^2} \cdot k \log\Par{\frac {\log k} \delta}} \text{ queries to } g_1 \text{ and } k \text{ evaluations of } h_2.
\end{align*}
The query depth used is $O(\log(k))$, and the computational depth used is $O((\log(d) + \log\log(\frac k \delta))\log(k))$.
\end{proposition}
\begin{proof}
Fix any two $i, j \in [k]$ with $i \neq j$. Our first step is to build a high-probability approximation subroutine for the value of $h_1(x_i) - h_1(x_j)$. To this end, observe that for $x^{(t)}_{i, j} \defeq (1 - t)x_i + tx_j$,
\begin{align*}h_1(x_j) - h_1(x_i) = \underbrace{\int_0^1 \inprod{\nabla h_1(x^{(t)}_{i, j})}{x_j - x_i} \dd t}_{\defeq I(i, j)} 
	= \E_{t \sim_{\textup{unif.}} [0,1]}\Brack{\inprod{\nabla h_1(x^{(t)}_{i, j})}{x_j - x_i}}.\end{align*}
Next, consider the estimator
\[Z(i, j) \defeq \inprod{g_1(x^{(t)}_{i, j})}{x_j - x_i},\text{ for } t \sim_{\textup{unif.}} [0, 1].\]
From the given assumptions, it is clear that $\E Z(i, j) = I(i, j)$ and 
\[\E Z(i,j)^2 \le \E\Brack{\norm{g_1(x^{(t)}_{i, j})}^2\norm{x_j - x_i}^2} \le L^2R^2.\]
Therefore, Chebyshev's inequality shows that averaging $\frac{4L^2 R^2}{\Delta^2}$ independent copies of $Z(i, j)$ produces a $\Delta$-additive approximation of $I(i, j)$ with probability $\frac 3 4$. A median of $O(\log(\frac {\log k} \delta))$ such independent averages then estimates $I(i, j)$ to additive error $\Delta$ with probability at least $1 - \frac \delta {\log k}$, by a Chernoff bound. The total computation required to produce this estimate for a pair $(i, j) \in [k] \times [k]$ is 
\[O\Par{\frac{L^2 R^2}{\Delta^2} \cdot \log\Par{\frac {\log k} \delta}} \text{ calls to } g_1.\]
Using two additional evaluations of $h_2$, we can thus estimate $h(x_j) - h(x_i)$ to additive error $\Delta$, with probability $\ge 1 - \frac \delta {\log k}$. To obtain the claim, we run a tournament on the elements of $S$ using our subroutine as an approximate comparator. Suppose that $k$ is a power of $2$ without loss of generality (otherwise we can duplicate $x_1$ appropriately), and initialize a complete binary tree on $2k - 1$ nodes (with depth $\log_2(k)$), placing elements of $S$ at the leaf nodes. We define the $i^{\text{th}}$ layer of the tree to be all nodes which are distance exactly $i$ from the leaf nodes (the leaf nodes themselves form layer $0$). Starting from layer $1$ and working upwards, for a node in layer $\ell$ with children $x_i$ and $x_j$, we compute $E(i, j)$, a $\Delta_\ell = \frac{\eps}{3} (\frac 4 3)^{-\ell}$-approximation to $h(x_i) - h(x_j)$, and promote the child with smaller estimated $h$ value (i.e., we promote $x_i$ if $E(i, j) \le 0$, and we promote $x_j$ otherwise). Assume without loss of generality that $x_1$ minimizes $h(x)$ over $S$. Conditioned on all estimates on $x_1$'s root-to-leaf path succeeding (which happens with probability $1-\delta$ since there are $\log k$ nodes), the minimum function value on level $\ell$ is at most $h(x_1) + \sum_{i\in[\ell]}  \frac{\eps}{3} (\frac 4 3)^{-i}$, and so the algorithm returns some node $y$ with  $h(y) \le \min_{i \in [k]} h(x_i) + \epsilon \le h(x^\star) + 2 \eps$. The complexity and correctness follow by setting $\Delta \gets \frac \eps {\log_2(k)}$. The total failure probability follows from a union bound, since there are at most $k - 1$ comparisons (as each comparison eliminates an element).

We now control the number of gradients computed by the algorithm. Level $\ell$ of the tree calls the estimation subroutine $k 2^{-\ell}$ times with failure probability $\frac{\delta}{\log k}$ and error $\frac{\eps}{3} (\frac 4 3)^{-\ell}$: summing over all layers gives a total gradient bound of 
\begin{align*}
O(1) \sum_{\ell \in [\log_2 k]} \frac{L^2 R^2 k 2^{-\ell}}{\eps^2 (\frac 4 3)^{-2\ell}}  \log \left(\frac{\log k} \delta\right) &= O\left(\frac{L^2 R^2}{\eps^2} k \log \left(\frac{\log k} \delta\right)\right)  \sum_{\ell \in [\log_2 k]}  \Par{\frac 8 9}^{\ell} \\
&= O\left(\frac{L^2 R^2}{\eps^2} k \log \left(\frac{\log k} \delta\right)\right).
\end{align*}

\end{proof}

\subsection{Putting it all together}\label{ssec:combine}

In this section, we combine the tools given in the previous sections to develop two high-probability optimization primitives, which will be used in Section~\ref{ssec:binary_search} in conjunction with a binary search to give our overall ball optimization oracle implementation. We now formally define the two types of oracles we require for implementing our binary search. Roughly speaking, the first type of oracle (Definition~\ref{def:phase_one}) is used to find a range of regularization amounts $\alpha$ such that the resulting regularized minimizers lie in a ball of radius $O(r)$. The second type of oracle (Definition~\ref{def:phase_two}) is then used to obtain accurate minima for our original constrained function. In the following definitions, for a fixed differentiable convex function $F$ and $\alpha > 0$, we let
\begin{equation}\label{eq:xsadef}\xsa \defeq \argmin_{x \in \R^d} F(x) + \frac \alpha 2 \norm{x}^2.\end{equation}

\begin{definition}\label{def:phase_one}
	We call $\oracle_1$ an $(r, \beta)$-\emph{phase-one oracle} for $F: \R^d \to \R$ if on input $\alpha \ge \beta$, following notation \eqref{eq:xsadef}, $\oracle_1$ returns $x$ satisfying 
	\[\norm{x - \xsa} \le \frac{r + \norm{\xsa}}{100}.\]
\end{definition}

\begin{definition}\label{def:phase_two}
	We call $\oracle_2$ a $(\Delta, r, \beta)$-\emph{phase-two oracle} for $F: \R^d \to \R$ if on input $\alpha \ge \beta$, following notation \eqref{eq:xsadef}, $\oracle_2$ returns $x$ satisfying 
	\[F(x) + \frac \alpha 2 \norm{x}^2 \le F(\xsa) + \frac \alpha 2 \norm{\xsa}^2 + \Delta.\]
\end{definition}

We specialize the following discussion to the specific context where $F$ is of the form 
\begin{equation}\label{eq:constrained_newton}
	\begin{gathered}
		\argmin_{x \in \ball(r)}\inprod{\nabla \frl(z)}{x} + \norm{x - z}^2_{\nabla^2 f_\rho(\0_d)} + \lam \norm{x - z}^2 \\
		= \argmin_{x \in \ball(r)}\inprod{\nabla f_\rho(z) - \lam z}{x} + \norm{x - z}^2_{\nabla^2 f_\rho(\0_d)} + \lam \norm{x}^2,
	\end{gathered}
\end{equation}		
where $\norm{z} \le r$ and $\rho \ge r$, $\lam > 0$. These constrained subproblems arise in an approximate Newton's method which we develop in Section~\ref{ssec:newton}. Formally, we define
\begin{equation}\label{eq:specific_f}F(x) \defeq \inprod{\nabla f_\rho(z) - \lam z}{x} + \norm{x - z}^2_{\nabla^2 f_\rho(\0_d)}.\end{equation}
We use two tools to boost constant-accuracy bounds to high probability. The first is the reduction in Proposition~\ref{prop:high_prob_reduction}, and the second is the following standard geometric aggregation method.

\begin{lemma}[Claim 1, \cite{KelnerLLST23}]\label{lem:agg}
	Let $\delta \in (0, 1)$ and $x \in \R^d$ be unknown, and let $\alg$ be an algorithm which returns $x' \in \R^d$ such that $\|x' - x\| \le \frac \Delta 3$ with probability $\ge \frac 2 3$ in $\mathcal{D}_{\alg}$ depth and $\mathcal{T}_{\alg}$ work. There is an algorithm which returns $y$ such that $\norm{y - x} \le \Delta$ with probability $\ge 1 - \delta$, using $O(\mathcal{D}_{\alg}+ \log(d))$ depth and $O(\mathcal{T}_{\alg} \cdot \log(\frac 1 \delta) + d\log^2(\frac 1 \delta))$ work.
\end{lemma}

We now state our oracle implementations and their guarantees.

\begin{lemma}\label{lem:oracle_one_impl}
	Let $F$ be defined as in \eqref{eq:specific_f}, assume $\rho \le \frac L \lam$, and let $\delta \in (0, 1)$. We can implement an $(r, 2\lam)$-phase-one oracle for $F$ which succeeds with probability $\ge 1 - \delta$ with
	\begin{gather*}O\Par{\dQuery + \log^2\Par{\frac{L}{\lam r}}\log(d)} \text{ depth,} \\
		\text{and } O\Par{\Par{\frac{L^2\log\Par{\frac{L}{\lam r}}}{\lam^2 r^2}}\log\Par{\frac 1 \delta} \cdot \tQuery + d\Par{\frac{L^2\log\Par{\frac{L}{\lam r}}}{\lam^2 r^2}}^{\omega - 1}\log\Par{\frac 1 \delta} + d\log^2\Par{\frac 1 \delta}} \text{ work.}
	\end{gather*}
\end{lemma}
\begin{proof}
	We first show how to obtain $x$ such that $\norm{x - \xsa} \le \frac{r + \norm{\xsa}}{300}$ with probability $\ge \frac 2 3$. By Markov's inequality and $\alpha$-strong convexity of $F(x) + \frac{\alpha}{2}\norm{x}^2$, it is enough to produce $x$ such that
	\[\E\Brack{\Par{F(x) + \frac \alpha 2 \norm{x}^2} - \Par{F(\xsa) + \frac \alpha 2 \norm{\xsa}}^2}\le \frac \alpha 6 \cdot \frac{(r + \norm{\xsa}^2)}{300^2}.\]
	In the context of Corollary~\ref{cor:specific_sgd_analysis} (with $\Lam \gets \alpha$), it suffices to take
	\[T = O\Par{\frac{\rho^2}{r^2} + \frac{L^2\log\Par{\frac{L}{\alpha r}}}{\alpha^2 r^2}} = O\Par{\frac{L^2\log\Par{\frac{L}{\alpha r}}}{\alpha^2 r^2}}.\]
	The conclusion follows from Corollary~\ref{cor:impl_sgd} (with $S \gets T$) and Lemma~\ref{lem:agg}.
\end{proof}

\begin{lemma}\label{lem:oracle_two_impl}
	Let $F$ be defined as in \eqref{eq:specific_f}, assume $\rho \in [r, \frac L \lam]$, and let $\delta \in (0, 1)$. For $\Delta \le \frac{\lam r^2}{100}$, we can implement a $(\Delta, r, \max(\alpha_{3r}, 2\lam))$-phase-two oracle for $f$ which succeeds with probability $\ge 1 - \delta$ with
	\begin{align*}
		O\Par{\log\log\Par{\frac 1 \delta} \cdot \dQuery + \frac{\lam r^2}{\Delta}\log\Par{\frac{L^2}{\lam\Delta}}\log^2\Par{\frac{L}{\lam r}}\log\Par{\frac d \delta}} \text{  depth,} \\
		\text{and } O\Par{\frac{L^2\log\Par{\frac{L^2}{\lam\Delta}}}{\lam\Delta}\log^3\Par{\frac 1 \delta} \cdot \tQuery + d\log^4\Par{\frac{L^2}{\delta\lam\Delta}} \cdot  \frac{\lam r^2}{\Delta} \cdot\Par{\frac{L^2}{\lam^2 r^2}}^{\omega -1}} \text{ work.}
	\end{align*}
\end{lemma}
\begin{proof}
	Since $\alpha \ge \alpha_{3r}$, we can produce a point $x$ with expected suboptimality $\frac \Delta 6$ to the function $F(x) + \frac \alpha 2 \norm{x}^2$ by calling Corollary~\ref{cor:specific_sgd_analysis} with
	\[T = O\Par{\frac{\alpha \rho^2}{\Delta} + \frac{L^2\log\Par{\frac{L^2}{\alpha\Delta}}}{\alpha\Delta}} = O\Par{\frac{L^2\log\Par{\frac{L^2}{\alpha\Delta}}}{\alpha\Delta}}.\]
	Therefore, by Markov's inequality $x$ has suboptimality $\frac \Delta 2$ with probability $\ge \frac 2 3$. Moreover, each $x$ which achieves this suboptimality has, by $\alpha$-strong convexity,
	\[\norm{x - \xsa} \le \sqrt{\frac{\Delta}{\alpha}} \le \frac r {10}.\]
	We run this algorithm $k = O(\log \frac 1 \delta)$ times, where the constant is large enough that Lemma~\ref{lem:agg} applies with probability $\ge 1 - \frac \delta 3$, and also $k \ge \log_3(\frac 3 \delta)$, so some run produces $x$ with suboptimality gap $\frac \Delta 2$ with probability $\ge 1 - \frac \delta 3$. Call $A = \{x_i\}_{i \in [k]}$ the set of output points, and let $x_{i^\star} \in S$ satisfy \[F(x_{i^\star}) + \frac \alpha 2 \norm{x_{i^\star}}^2 - F(\xsa) - \frac \alpha 2 \norm{\xsa}^2 \le \frac \Delta 2.\] 
	By Lemma~\ref{lem:agg}, we obtain $\bx$ with $\norm{\bx - \xsa} \le 3\sqrt{\Delta/\alpha}$. Let $B \subseteq A$ be the elements of $A$ with $\norm{\bx - x} \le 4\sqrt{\Delta/\alpha}$, which contains $x_{i^\star}$ by definition. Then for any $x, x' \in B$, we have
	\[\norm{x - x'} \le 8\sqrt{\frac \Delta \alpha}.\]
	Moreover, since all points in $B$ lie at distance $\le \frac {2r} 5$ from $\xsa$, their norms are all at most $4r$. Since $\norm{z} \le r$ by assumption, Lemma~\ref{lem:second_moment} shows we can implement an unbiased estimator for the gradient of the implicit part of \eqref{eq:specific_f}, i.e., $\nabla f_\rho(z) + 2\nabla^2 f_\rho(\0_d)(x - z)$, with second moment $O(L^2)$, for any $x$ in the convex hull of $B$. We therefore can apply Proposition~\ref{prop:high_prob_reduction} with $\eps \gets \frac \Delta 2$ to obtain an element of $B$ with suboptimality gap $\le \Delta$ with probability $\ge 1 - \frac \delta 3$, within the stated complexities. We can check that all other steps also fall within the stated complexities, using Corollary~\ref{cor:impl_sgd} with $S \gets \frac{L^2}{\alpha \lam r^2}$.
\end{proof} %
\section{Parallel stochastic convex optimization}
\label{sec:framework}
\newcommand{\xaopt}{x_\alpha^\star}

In this section, we prove Theorem~\ref{thm:main} by using the results of Section~\ref{sec:quadratic} to implement the ball optimization oracles required by Proposition~\ref{prop:ball_accel}. Our reduction from (constrained) ball optimization to the (unconstrained) quadratic problems considered by Section~\ref{sec:quadratic} proceeds in two steps.

\begin{enumerate}
    \item In Section~\ref{ssec:newton}, we show how to use Hessian stability of the ball optimization oracle subproblems (Corollary~\ref{cor:conv_reg_stable}) to efficiently solve these problems using an approximate Newton's method. 
    \item The subproblems required by our method in Section~\ref{ssec:newton} are constrained optimization problems, which are almost compatible with our tools in Section~\ref{sec:quadratic}. In Section~\ref{ssec:binary_search}, we develop a simple binary search procedure, inspired by a procedure in \cite{JambulapatiRT23}, which reduces each constrained optimization problem to a small number of unconstrained stochastic optimization problems.
\end{enumerate}

Finally, we show how to combine the pieces and give our proof of Theorem~\ref{thm:main} in Section~\ref{ssec:full_algo}.

\subsection{Approximate Newton's method}\label{ssec:newton}

In this section, we state and analyze an approximate variant of a constrained Newton's method under Hessian stability, patterned off classical analyses of gradient descent in a quadratic norm.

\begin{algorithm2e}
	\caption{$\CN(\lam, T, x_0, f, \phi)$}
	\label{alg:quad_newton}
	\DontPrintSemicolon
		\codeInput Positive definite $\ma \in \R^{d \times d}$, $T \in \N$,  $x_0 \in \xset$, differentiable $f: \xset \to \R^d$, $\phi > 0$ \;
		\For{$0 \le t < T$}{
		$x_{t + 1} \gets $ any (randomized) point in $\xset$ satisfying
  \[\E \inprod{\nabla f(x_t)}{x_{t + 1}} + \norm{x_{t + 1} - x_t}^2_{\ma} \le \min_{x \in \xset}\Brace{\inprod{\nabla f(x_t)}{x } + \norm{x - x_t}^2_{\ma}} + \frac \phi {20} \]\label{line:min_quad}
		}
    \codeReturn $x_T$
\end{algorithm2e}

\begin{lemma}\label{lem:quad_newton}
Let $\phi > 0$, let $f: \xset \to \R$ be twice-differentiable for convex $\xset \subset \R^d$, and let $x^\star \defeq \argmin_{x \in \xset} f(x)$. Assume that $\half \ma \preceq \nabla^2 f(x) \preceq 2\ma \text{ for all } x \in \xset$, 
for positive definite $\ma \in \R^{d \times d}$. Algorithm~\ref{alg:quad_newton} with $T \gets O(\log \frac{f(x_0) - f(x^\star)}{\phi})$ returns $x_T \in \xset$ satisfying $\E f(x_T) \le f(x^\star) + \phi$. 
\end{lemma}
\begin{proof}
Throughout the proof, let $\Phi_t \defeq \E f(x_t) - f(x^\star)$, so our goal is to show $\Phi_T \le \phi$. We first observe that, conditioning on $x_t$, and letting $x^{(s)} \defeq (1 - s)x_t + sx^\star$,
\begin{align*}
\E f(x_{t + 1}) &\le \E \Brack{f(x_t) + \inprod{\nabla f(x_t)}{x_{t + 1} - x_t} + \norm{x_{t + 1} - x_t}_{\ma}^2} \\
&\le \min_{s \in [0, 1]} \Brace{f(x_t) + \inprod{\nabla f(x_t)}{x^{(s)} - x_t} + \norm{x^{(s)} - x_t}_{\ma}^2} + \frac \phi {20} \\
&\le \min_{s \in [0, 1]} \Brace{f(x^{(s)}) + s^2 \norm{ x_t - x^\star}_{\ma}^2} + \frac \phi {20} \\
&\le \min_{s \in [0, 1]} \Brace{f(x^{(s)}) + 4s^2 \Phi_t} + \frac \phi {20}.
\end{align*}
Above, the first line used a second-order Taylor expansion and our assumption $\nabla^2 f(x) \preceq 2\ma$ pointwise, the second line used the definition of $x_{t + 1}$, the third used convexity, and the last used first-order optimality of $x^\star$ as well as our assumption $\half \ma \preceq \nabla^2 f(x)$ pointwise which implies that
\[\frac 1 4 \norm{x_t - x^\star}_{\ma}^2 \le \inprod{\nabla f(x^\star)}{x_t - x^\star} + \frac 1 4 \norm{x_t - x^\star}_{\ma}^2 \le \Phi_t. \]
Subtracting $f(x^\star)$ from both sides and using convexity once more yields
\begin{align*}
\E \Phi_{t + 1} \le \min_{s \in [0, 1]} \Brace{(1 - s)\Phi_t + 4s^2 \Phi_t} + \frac \phi {20} = \frac{15}{16} \Phi_t + \frac \phi {20}. 
\end{align*}
Recursively applying for $T$ iterations, and using $\frac{1}{20} \sum_{i = 0}^{\infty} (\frac{15}{16})^i \leq1$, yields the conclusion.
\end{proof}

Lemma~\ref{lem:quad_newton} and Corollary~\ref{cor:conv_reg_stable} show that to implement a ball optimization oracle for the function $f = \frlbx$ (defined in \eqref{eq:frldef}) over sufficiently small radii, it suffices to implement Line~\ref{line:min_quad} of Algorithm~\ref{alg:quad_newton} logarithmically many times. Concretely, when $\xset = \ball(r)$ and $f = \frlbx$, Line~\ref{line:min_quad} requires a $\Theta(\phi)$-approximate minimizer to a problem of the form
in \eqref{eq:constrained_newton},
for $z \in \ball(r)$ given by the algorithm. These are exactly the problems which our tools in Section~\ref{sec:quadratic} can approximately solve, except they are hard-constrained. We show how to lift the constraints via a regularized binary search in Section~\ref{ssec:binary_search}.

\subsection{Ball optimization oracles via binary search}\label{ssec:binary_search}

In this section, we provide a binary search strategy for approximately solving the constrained optimization problem \eqref{eq:constrained_newton}, by binary searching on a Lagrange multiplier for the constraint. To begin, we require the following claims on the minima of regularized convex functions.

\begin{lemma}\label{lem:reg_minima_properties}
Let $F: \R^d \to \R$ be a twice-differentiable convex function satisfying $\norm{\nabla F(\0_d)} \le L$, and for all $\alpha \in \R_{\ge 0}$, let $\xsa \defeq \argmin_{x \in \R^d} F(x) + \frac \alpha 2 \norm{x}^2$. We have the following claims.
\begin{enumerate}
    \item For all $0 < \alpha < \alpha'$, $\norm{\xsa} > \norm{x^\star_{\alpha'}}$. \label{item:shrinking}
    \item For all $0 < \alpha < \alpha'$, $\norm{\xsa - x^\star_{\alpha'}} \le \norm{\xsa}\log(\frac{\alpha'}{\alpha})$. \label{item:log_shrinking}
    \item If $\alpha \ge \frac{4L}{r}$, $\norm{\xsa} \le \frac r 2$. \label{item:upper_bound}
\end{enumerate}
\end{lemma}
\begin{proof}
The optimality conditions on $\xsa$ show that $\nabla F(\xsa) = -\alpha \xsa$, so differentiating in $\alpha$,
\[\nabla^2 F(\xsa) \Par{\frac{\dd}{\dd \alpha} \xsa} = -\xsa - \alpha \cdot \frac{\dd}{\dd \alpha} \xsa \implies \frac{\dd}{\dd \alpha} \xsa = -\Par{\nabla^2 F(\xsa) + \alpha\id_d}^{-1} \xsa. \]
Therefore, for any $0 < \alpha < \alpha'$, we have Item~\ref{item:shrinking}, as convexity of $F$ shows
\begin{align*}
\half\norm{x^\star_{\alpha'}}^2 - \half\norm{\xsa}^2 = \int_\alpha^{\alpha'} \Par{-\norm{x^\star_t}_{(\nabla^2 F(x^\star_t) + t\id_d)^{-1}}} \dd t \le 0.
\end{align*}
Now, by using the triangle inequality and Item~\ref{item:shrinking}, Item~\ref{item:log_shrinking} follows:
\begin{align*}
\norm{\xsa - x^\star_{\alpha'}} \le \int_\alpha^{\alpha'} \norm{(\nabla^2 F(x^\star_t) + t\id_d)^{-1} x^\star_t} \dd t
\le \int_\alpha^{\alpha'} \frac 1 t \norm{x^\star_t} \dd t \le \norm{\xsa}\log\Par{\frac{\alpha'}{\alpha}}.
\end{align*}
Finally, to see Item~\ref{item:upper_bound}, note that for $\alpha \ge \frac{4L}{r}$,
\begin{align*}
F(\0_d) \ge F(\xsa) + \frac{\alpha}{2}\norm{\xsa}^2 
\ge F(\0_d) + \inprod{\nabla F(\0_d)}{\xsa} + \frac{\alpha}{2}\norm{\xsa}^2 
\ge F(\0_d) - L\norm{\xsa} + \frac{\alpha}{2}\norm{\xsa}^2.
\end{align*}
Rearranging and applying our lower bound on $\alpha$ yields $\norm{\xsa} \le \frac r 2$ as claimed.
\end{proof}

In light of Lemma~\ref{lem:reg_minima_properties}, in the remainder of the section we fix a differentiable convex function $F$, and develop a generic framework for approximately solving, for a parameter $\lam > 0$,
\[\argmin_{x \in \ball(r)} F(x) + \lam \norm{x}^2.\]
We follow the notation \eqref{eq:xsadef} throughout for brevity, so the above minimizer is denoted $x^\star_{2\lam}$. 
For convenience, for any $t \in [0, \norms{x^\star_{2\lam}}]$, we also use $\alpha_t$ to denote the unique value of $\alpha \in [2\lam, \infty)$ such that $\norms{x^\star_{\alpha_t}} = t$, where uniqueness and existence follows from Lemma~\ref{lem:reg_minima_properties} and $x^\star_\infty = \0_d$. 

At the end of Section~\ref{ssec:combine}, we gave implementations of a phase-one oracle and a phase-two oracle for $F$ in \eqref{eq:specific_f} (see Lemmas~\ref{lem:oracle_one_impl} and~\ref{lem:oracle_two_impl}). We now apply Definitions~\ref{def:phase_one} and~\ref{def:phase_two} to implement our binary search, stated formally in the following and with pseudocode provided in Algorithm~\ref{alg:binary}.

\newcommand{\BinarySearch}{\textsf{BinarySearch}}
\begin{algorithm2e}[h]
	\caption{$\BinarySearch(\lambda, r, \Delta, L, \mathcal{O}_1, \mathcal{O}_2)$}
	\label{alg:binary}
	\DontPrintSemicolon
		\codeInput $\lam, r, \Delta, L \in \R_{> 0}$, $\oracle_1$, an $(r, 2\lam)$-phase-one oracle (Definition~\ref{def:phase_one}) for differentiable convex $f: \R^d \to \R$ satisfying $\norm{f(\0_d)} \le L$, $\oracle_2$, a $(\frac \Delta 2, r, \max(\alpha_{3r}, 2\lam))$-phase-two oracle for $f$\;
		\tcp{Start phase one.}
        $u \gets 2\lam$\; \label{line:phase1_start}
		\lWhile{$\norm{\mathcal{O}_1(u)} > 2.5 r$}{
		$u \gets 2u$ \label{line:phase0}
		} 
        \lIf{$u = 2\lambda$}{
        $\alpha' \gets 2\lam$ \label{line:phase1early}
        }
        \uElse{
        $\ell \gets \frac u 2$\;\label{line:ul_init}
        \While{$\mathbf{true}$ \label{line:phase1}}{
        $m \gets \sqrt{u\ell}$\;
        \lIf{$\norm{\mathcal{O}_1(m)} \in [2.1r, 2.9 r]$}{
        $\alpha' \gets m$ and \textbf{break} \label{line:phase1end}
        }
        \lElseIf{$\norm{\mathcal{O}_1(m)} > 2.9r$}{
        $\ell \gets m$
        }
        \lElse{
        $u \gets m$
        }
        }
        }
\BlankLine
\tcp{Start phase two.}
$\ell \gets \alpha'$\label{line:phase2start}\;
$u \gets \frac{4L}{r} + 2\lambda$\;
\While{$\frac{u}{\ell} > 1+\frac{\Delta}{10(Lr + \lambda r^2)}$ \label{line:phase2}}{
$m \gets \sqrt{u\ell}$\;
\lIf{$\norm{\mathcal{O}_2(m)} > r$}{
$\ell \gets m$
}
\lElse{
$u \gets m$
}
}
$x_1 \gets \mathcal{O}_1(\ell)$, $x_2 \gets \mathcal{O}_2(u)$\;
\lIf{$\ell = 2\lambda$}{
\codeReturn $x_1$
}
\codeReturn $x_{\text{out}} \gets (1-t)x_1 + t x_2$, where $t \in [0, 1]$ is chosen so $\norm{x_{\text{out}}} = r$\label{line:xout}
\end{algorithm2e}

\begin{proposition}
\label{prop:binsearch}
Let $F$ be a differentiable convex function satisfying $\norm{\nabla F(\0_d)} \leq L$. Let $\lam, \Delta, r \in \R_{> 0}$ with $\Delta \le \frac{\lam r^2}{100}$. Algorithm~\ref{alg:binary} computes $x \in \ball(r)$ satisfying
\[
F(x) + \lambda \norm{x}^2 \leq \min_{\norm{y} \leq r} F(y) + \lambda \norm{y}^2 + \Delta.
\]
Algorithm~\ref{alg:binary} makes at most $O(\log \frac{L}{\lam r})$ calls to $\oracle_1$, and $O(\log \frac{Lr + \lam r^2}{\Delta})$ calls to $\oracle_2$.
\end{proposition}
\begin{proof}
We start with a correctness proof, and bound the number of oracle calls at the end. Because the specifications of $\oracle_1$ and $\oracle_2$ do not preclude returning different answers on multiple calls with the same $\alpha$, throughout the proof to alleviate burdensome notation, we assume that if an oracle is called twice with the same $\alpha$, it gives the same result (e.g., the result of the first call).

We begin by analyzing the first phase of the algorithm, starting from Line~\ref{line:phase1_start} and ending before Line~\ref{line:phase2start}. By the criterion in the while loop on~\cref{line:phase0}, $u$ satisfied $\norm{\oracle_1(u)} \leq 2.5 r$, so
\[
\norm{\mathcal{O}_1(u) - x^\star_{u}} \leq \frac{1}{100} ( r + \norm{x^\star_{u}}),
\]
which implies that for the value of $u$ after the while loop ends,
\[
\norm{x^\star_{u}} - 2.5 r \leq \norm{x^\star_{u}} - \norm{\mathcal{O}_1(u)} \leq \frac{1}{100} ( r + \norm{x^\star_{u}}) \implies \norm{x^\star_{u}} \leq \frac{100}{99} \cdot (2.51 r) < 3r.
\]
Next, we claim that at the conclusion of phase one, either $\alpha' = 2\lam$ and $\oracle_1(2\lam) \le 2.5r$, or $\alpha'$ has
\begin{equation}\label{eq:alphaprime}\norm{x^\star_{\alpha'}} \in \Brack{2r, 3r}.\end{equation}
The first case is obvious from Line~\ref{line:phase1early}. Otherwise, for the values of $\ell, u$ on Line~\ref{line:ul_init}, we have $\ell \ge 2\lam$ and $\norm{\oracle_1(u)} \le 2.5r < \norm{\oracle_1(\ell)}$. When the while loop breaks on Line~\ref{line:phase1end}, we have $\norm{\oracle_1(m)} \in [2.1r, 2.9r]$, which yields \eqref{eq:alphaprime} (since $\alpha' = m$ in this case) due to the following derivations:
\begin{align*}
\norm{x^\star_m} - 2.9r \le \norm{x^\star_m - \oracle_1(m)} \le \frac 1 {100}\Par{r + \norm{x^\star_m}} \implies \norm{x^\star_m} \le \frac{100}{99} \cdot (2.91r) < 3r    ,\\
2.1r - \norm{x^\star_m} \le \norm{x^\star_m - \oracle_1(m)} \le \frac 1 {100}\Par{r + \norm{x^\star_m}} \implies \norm{x^\star_m} \ge \frac{100}{101} \cdot (2.09r) > 2r   .
\end{align*}
Both bounds used the triangle inequality. This concludes our correctness analysis of Phase 1.

We now analyze correctness of phase two. By Item~\ref{item:upper_bound} of Lemma~\ref{lem:reg_minima_properties}, we have $\norm{\oracle_2(u)} \le \norm{x^\star_u} + \norm{x^\star_u - \oracle_2(u)} \le r$ on Line~\ref{line:phase2start}, where we used strong convexity of $F(x) + \frac \alpha 2 \norm{x}^2$ to bound $\norm{x^\star_u - \oracle(u)} \le \frac r 2$. Thus, inspecting the while loop starting on Line~\ref{line:phase2}, we preserve the invariants:
\[\ell < u,\; \norm{\oracle_2(u)} \le r,\text{ and either } \norm{\oracle_2(\ell)} > r, \text{ or } \ell = 2\lam.\]
In particular, if $\norm{\oracle_2(\ell)} \le r$, it must be that $\alpha' = \ell$ (i.e.\ $\ell$ never updated), but if $\alpha' \neq 2\lam$ then this is impossible by \eqref{eq:alphaprime}. 
Hence, when the while loop on \cref{line:phase2} terminates, the values $\ell, u$ associated with $x_1$, $x_2$ satisfy $u \in [\ell, (1 + \frac{\Delta}{10(Lr + \lam r^2)})\ell]$, $\norm{x_2} \le r$, and we are in one of the following cases.
\begin{enumerate}
    \item $\ell = 2\lam$ and $\oracle_2(u) \le r$.\label{item:case1}
    \item $x_1 = \oracle_2(\ell)$ has $\norm{x_1} > r$.\label{item:case2}
\end{enumerate}
In Case~\ref{item:case1}, let $y \defeq \argmin_{y \in \ball(r)} F(y) + \lam \norm{y}^2$. Then by the guarantees of $\oracle_2$,
\begin{align*}
F(x_1) + \lam \norm{x_1}^2 &\le F(x_1) + \frac u 2 \norm{x_1}^2 \le F(y) + \frac u 2 \norm{y}^2 + \frac \Delta 2 \\
&\le F(y) + \lam\norm{y}^2 + \Par{\frac u 2 - \lam}r^2 + \frac \Delta 2 \le F(y) + \lam\norm{y}^2 + \Delta,
\end{align*}
where we used $\frac u 2 - \lam \le \frac{\Delta}{10\lam r^2} \cdot \frac \lam 2 \le \frac \Delta {2r^2}$. On the other hand, in Case~\ref{item:case2},
recalling $\norm{x^\star_{\ell}}, \norm{x^\star_{u}} \leq 3r$ by the guarantees of phase one, and letting $\xout \defeq (1 - t)x_1 + t x_2$ as in Line~\ref{line:xout},
\[
(1-t) F(x_1) + t F(x_2) + \frac{(1-t) \ell} 2 \norm{x_1}^2 + \frac{tu} 2 \norm{x_2}^2 \leq F(y) + \frac{u} 2 \norm{y}^2 + \frac \Delta 2,
\]
for every $y \in \R^d$, by the definition of $\oracle_2$. Now, letting $y \defeq \argmin_{y \in \ball(r)} F(y) + \lam\norm{y}^2$,
\begin{align*}
F(\xout) + \frac \ell 2 \norm{\xout}^2 &\leq (1-t) F(x_1) + t F(x_2) +  \frac{(1-t) \ell} 2 \norm{x_1}^2 + \frac{tu} 2 \norm{x_2}^2 \\
&\leq F(y) + \frac u 2 \norm{y}^2 + \frac \Delta 2.
\end{align*}
Additionally, note that if we are in Case~\ref{item:case2}, then $\norm{y} = r$. To see this, suppose for contradiction that $\norm{y} < r$, which means $\norms{x^\star_{2\lam}} < r$. If $\alpha' > 2\lam$, then \eqref{eq:alphaprime} and Item~\ref{item:shrinking} of Lemma~\ref{lem:reg_minima_properties} give a contradiction. Otherwise, $\alpha' = 2\lam$, but then $\norms{x^\star_{2\lam}} < r$ contradicts the statement before \eqref{eq:alphaprime} since $\oracle_1(2\lam) \le 2.5r$ cannot happen. Hence, $\norm{y} = \norm{\xout} = r$, and correctness in Case~\ref{item:case2} follows from
\begin{align*}
F(\xout) + \lam\norm{\xout}^2 &= F(\xout) + \lam\norm{y}^2 \\
&\le F(y) + \lam\norm{y}^2 + \frac{(u - \ell) r^2} 2 + \frac \Delta 2 \le F(y) + \lam\norm{y}^2 + \Delta,
\end{align*}
where $u - \ell \le \frac{\Delta}{10(Lr + \lam r^2)} \cdot u \le \frac{\Delta}{r^2}$, since $u \le \frac{4L}{r} + 2\lam$. This completes the correctness proof.

We now bound the number of calls to $\mathcal{O}_1, \mathcal{O}_2$. By Item~\ref{item:upper_bound} of Lemma~\ref{lem:reg_minima_properties} and \eqref{eq:alphaprime}, it is clear the number of times Line~\ref{line:phase0} occurs is $O(\log \frac{L}{\lam r})$. Next, consider the loop on \Cref{line:phase1} until  \Cref{line:phase1end} is hit. We claim the loop must break if $\log\frac u \ell \le \frac 1 {100}$, which means the loop can only run $O(1)$ times, because $\log \frac u \ell$ halves each time the loop is run, and $\frac u \ell = 2$ initially. To see our claim, for any $\alpha$,
\begin{equation}\label{eq:alphabound}
\begin{aligned}
\norm{\oracle_1(\alpha)} - \norm{\xsa} \le \frac {r + \norm{\xsa}} {100} &\implies \frac{100}{101}\norm{\oracle_1(\alpha)} - \frac r {101} \le \norm{\xsa}, \\
\norm{\xsa} - \norm{\oracle_1(\alpha)} \le \frac {r + \norm{\xsa}} {100} &\implies \frac{100}{99}\norm{\oracle_1(\alpha)} + \frac r {99} \ge \norm{\xsa},
\end{aligned}
\end{equation}
which follow from the definition of $\oracle_1$. Further, by Item~\ref{item:log_shrinking} of Lemma~\ref{lem:reg_minima_properties}, supposing $\log\frac u \ell \le \frac 1 {100}$,
\[\norm{x^\star_u - x^\star_\ell} \le \frac {\norm{x^\star_\ell}} {100} \le \frac 1 {100} \Par{\frac{100}{99}\norm{\oracle_1(\ell)} + \frac r {99}},\]
where we used the second bound in \eqref{eq:alphabound}. Combining with \eqref{eq:alphabound}, we have
\begin{align*}
\frac 1 {100} \Par{\frac{100}{99}\norm{\oracle_1(\ell)} + \frac r {99}} &\ge \norm{x^\star_\ell} - \norm{x^\star_u} \ge \Par{\frac{100}{101}\norm{\oracle_1(\ell)} - \frac r {101}} - \Par{\frac{100}{99}\norm{\oracle_1(u)} + \frac r {99}} \\
\implies \frac{100}{99}\norm{\oracle_1(u)} + \frac r {33} &\ge \Par{\frac{100}{101} - \frac 1 {99}}\norm{\oracle_1(\ell)},
\end{align*}
which is a contradiction since $\norm{\oracle_1(u)} < 2.1r$ and $\norm{\oracle_1(\ell)} > 2.9r$ until termination.

Finally, consider the loop starting on Line~\ref{line:phase2}. At the beginning, we have $\frac u \ell = O(\frac{L}{\lam r} + 1)$, and $\log \frac u \ell$ halves each time the loop is run. Therefore, $\oracle_2$ is called $O(\frac{Lr + \lam r^2}{\Delta})$ times as claimed.
\end{proof}

We now combine Proposition~\ref{prop:binsearch} with Lemmas~\ref{lem:oracle_one_impl} and~\ref{lem:oracle_two_impl} to give our parallel ball optimization oracle.

\begin{proposition}\label{prop:parallel_ball_oracle}
Define $f_\rho$ as in Definition~\ref{def:gconv}, where $f$ is in the setting of Problem~\ref{prob:sco}. Let $\lam, r \in \R_{> 0}$ satisfy $r \le \frac{\rho}{6} \cdot \log^{-\half}(\frac{2L}{\lam\rho})$ and $\rho \le \frac L \lam$. For any $\phi \in (0, \frac{\lam r^2}{100}]$, we can implement a $(\phi, \lam, r)$-ball optimization oracle (Definition~\ref{def:boo}) for $f_\rho$ with
\begin{align*}
O\Par{\log\Par{\frac{Lr}{\phi}}\log\log\Par{\frac{Lr}{\phi}} \cdot \dQuery + \frac{\lam r^2}{\phi}\log^4\Par{\frac{L^2}{\lam\phi}}\log\Par{\frac{dL^2}{\lam\phi}}}\text{ depth,} \\
\text{and } O\Par{\frac{L^2}{\lam\phi}\log^5\Par{\frac{L^2}{\lam\phi}}\cdot \tQuery + d\log^5\Par{\frac{L^2}{\lam \phi}} \cdot \frac{\lam r^2}{\phi} \cdot \Par{\frac{L^2}{\lam^2 r^2}}^{\omega - 1}} \text{ work.} 
\end{align*}
\end{proposition}
\begin{proof}
Throughout, assume $\bx = \0_d$ in the definition of the ball optimization oracle, which is without loss of generality because shifting by a constant vector does not affect the assumptions in Problem~\ref{prob:sco}. Also, define $\frlbx$ as in \eqref{eq:frldef}, where $\bx = \0_d$. We give an algorithm which always returns a point $x$ in $\ball(r)$, and such that $x$ has suboptimality gap $\frac \phi 2$, except with probability $\delta \defeq \frac{\phi}{2Lr + \lam r^2}$. Because $f$ is $L$-Lipschitz by Jensen's inequality on the moment bound in Problem~\ref{prob:sco}, so is $f_\rho$ by Fact~\ref{fact:gconv}. Therefore the range of $\frlbx$ over $\ball(r)$ is at most $Lr + \lam r^2$, and the expected suboptimality gap is
\[(1 - \delta) \cdot \frac \phi 2 + \delta \cdot \Par{Lr + \frac {\lam r^2} 2} \le \phi,\]
as required. To implement this algorithm, we first apply Lemma~\ref{lem:quad_newton} and Corollary~\ref{cor:conv_reg_stable}, which show that it suffices to solve $T = O(\log\frac{Lr + \lam r^2}{\phi}) = O(\log \frac{Lr} \phi)$ problems of the form, for some $z \in \ball(r)$,
\[\underbrace{\inprod{\nabla f_\rho(z) - \lam z}{x} + \norm{x - z}^2_{\nabla^2 f_\rho(\0_d)}}_{\defeq F(x)} + \lam\norm{x}^2,\]
each to error $\Delta \defeq \frac \phi {40}$ (see \eqref{eq:constrained_newton} for the derivation). Note that 
\[\norm{\nabla F(\0_d)} = \norm{\nabla f_\rho(z) - \lam z - 2\nabla^2 f_\rho(\0_d) z} \le 2L + \lam r,\]
because $f_\rho$ is $L$-Lipschitz and $\frac L \rho$-smooth (Fact~\ref{fact:gconv}). Finally, let $Z \defeq O(\log \frac{Lr}{\phi})$ be the total number of oracle calls to $\oracle_1$ or $\oracle_2$ used by Proposition~\ref{prop:binsearch}. We implement each oracle using either Lemma~\ref{lem:oracle_one_impl} or Lemma~\ref{lem:oracle_two_impl} appropriately, with failure probability set to $\frac \delta Z$, and the conclusion follows.
\end{proof}

We also note that we can achieve a computational depth-complexity tradeoff in Proposition~\ref{prop:parallel_ball_oracle} by choosing different values of $S$ in Corollary~\ref{cor:impl_sgd}, than was used in Lemma~\ref{lem:oracle_two_impl}. As stated, Lemma~\ref{lem:oracle_two_impl} uses Corollary~\ref{cor:impl_sgd} by applying our parallel implementation $S$ iterations at a time, where $S = \frac{L^2}{\alpha\lam r^2}$ is the number of iterations required to achieve $\approx \lam r^2$ error. By instead choosing a larger error $C \cdot \lam r^2$ for a parameter $C \in [1, \frac{L^2}{\lam^2 r^2}]$, which induces $S = \frac{L^2}{C\alpha\lam r^2}$, we can obtain improved total work bounds at the cost of larger computational depth. In particular, Corollary~\ref{cor:impl_sgd} has a computational depth scaling linearly in the parameter $C$, and a computational complexity scaling as $C^{2 - \omega}$; all logarithmic terms are unnaffected, since $C \le \frac{L^2}{\lam^2 r^2}$. We summarize this observation in the following.

\begin{corollary}\label{cor:tradeoff_ball_oracle}
In the context of Proposition~\ref{prop:parallel_ball_oracle}, for any $C \in [1, \frac{L^2}{\lam^2 r^2}]$, we can implement a $(\phi, \lam, r)$-ball optimization oracle (Definition~\ref{def:boo}) for $f_\rho$ with
\begin{gather*}
O\Par{\log\Par{\frac{Lr}{\phi}}\log\log\Par{\frac{Lr}{\phi}} \cdot \dQuery + \frac{C\lam r^2}{\phi}\log^4\Par{\frac{L^2}{\lam\phi}}\log\Par{\frac{dL^2}{\lam\phi}}}\text{ depth,} \\
\text{and } O\Par{\frac{L^2}{\lam\phi}\log^5\Par{\frac{L^2}{\lam\phi}} \cdot \tQuery + d\log^5\Par{\frac{L^2}{\lam \phi}} \cdot \frac{\lam r^2}{\phi} \cdot \Par{\frac{L^2}{\lam^2 r^2}}^{\omega - 1} \cdot \Par{\frac 1 C}^{\omega - 2}} \text{ work.} 
\end{gather*}
\end{corollary}

\subsection{Proof of Theorem~\ref{thm:main_formal}}\label{ssec:full_algo}

In this section, we prove our main result, Theorem~\ref{thm:main_formal}, by combining the ball acceleration framework in Proposition~\ref{prop:ball_accel} with our parallel ball optimization oracle implementation in Proposition~\ref{prop:parallel_ball_oracle}.

\restatemainformal*
\begin{proof}
Throughout, let $\rho \defeq \frac \eps {2L\sqrt{d}}$, and let $x^\star$ minimize $f$ over $\ball(R)$. We optimize $f_\rho$ to expected error $\frac \eps 2$, yielding the conclusion via Observation~\ref{obs:conv_suffices}. Let $K = \Theta(d^{\frac 1 3}\kappa^{\frac 2 3}\log^{\frac 1 3}(d\kappa))$, and choose
\begin{equation}\label{eq:lams_r_def}\lams = \Theta\Par{\frac{\eps \kappa^{\frac 4 3}d^{\frac 2 3}}{R^2}\log^2(d\kappa)},\; r = \Theta\Par{\frac{\rho}{\sqrt{\log\Par{d\kappa}}}},\end{equation}
to be compatible with the parameters in Proposition~\ref{prop:ball_accel}, such that $r \le \frac \rho 6 \cdot \log^{-\half}(\frac{2\cba L}{\lams \rho})$, following the notation in Proposition~\ref{prop:ball_accel}. This implies that Corollary~\ref{cor:conv_reg_stable} holds for every choice of  $\lam \ge \frac{\lams}{\cba}$ used in ball optimization oracles by Proposition~\ref{prop:ball_accel}. Therefore, assuming $\cba \ge 100$ without loss of generality, we can use Proposition~\ref{prop:parallel_ball_oracle} to implement every ball optimization oracle.

To bound the query depth, we apply Proposition~\ref{prop:parallel_ball_oracle} for each of the $O(K\log^3(d\kappa))$ ball optimization oracles required. To bound the query complexity, we have the claim from
\begin{align*}
O\Par{K\log^3(d\kappa) \cdot \frac{L^2}{\lams^2 r^2}\log^5(d\kappa)} &= O\Par{\kappa^2 \log^{\frac{16}{3}}(d\kappa)},\\
\sum_{j \in [\lceil \log_2 K + \cba \rceil]} O\Par{2^{-j} K\log\Par{d\kappa} \cdot \frac{2^j L^2}{\lams^2 r^2}\log^7\Par{d\kappa}} &= O\Par{\frac{KL^2}{\lams^2 r^2}\log^9\Par{d\kappa}} \\
&= O\Par{\kappa^2\log^{\frac{19}{3}}(d\kappa)}.
\end{align*}
We also require one query per ball optimization oracle, so there is an additive $K\log^3(d\kappa)$ term. To bound the computational depth, we perform a similar calculation using Proposition~\ref{prop:parallel_ball_oracle}:
\begin{align*}
O\Par{K\log^3(d\kappa) \cdot \log^5(d\kappa)} &= O\Par{d^{\frac 1 3}\kappa^{\frac 2 3}\log^{\frac{25}{3}}(d\kappa)},\\
\sum_{j \in [\lceil \log_2 K + \cba \rceil]} O\Par{2^{-j} K\log\Par{d\kappa} \cdot 2^j\log^7\Par{d\kappa}} &= O\Par{K\log^9(d\kappa)} = O\Par{d^{\frac 1 3}\kappa^{\frac 2 3}\log^{\frac{28}{3}}(d\kappa)}.
\end{align*}
Finally, for the computational complexity, we have (using the bound $\omega \ge 2$)
\begin{align*}
O\Par{K\log^3(d\kappa) \cdot d\log^5(d\kappa) \cdot \Par{\frac{L^2}{\lams^2 r^2}}^{\omega - 1}} &= O\Par{d^{\frac 4 3}\kappa^{\frac 2 3}\log^{\frac{34}{3} -3\omega}(d\kappa) \cdot \Par{\frac{\kappa^{\frac 4 3}}{d^{\frac 1 3}}}^{\omega - 1}} \\
&=O\Par{d^{\frac{5-\omega}{3}}\kappa^{\frac{4\omega - 2}{3}}\log^{\frac{16}{3}}(d\kappa)} ,
\end{align*}
and
\begin{align*}
\sum_{j \in [\lceil \log_2 K + \cba \rceil]} O\Par{2^{-j} K\log\Par{d\kappa} \cdot 2^j \cdot d\log^{7}(d\kappa)\Par{\frac{L^2}{\lams^2 r^2}}^{\omega - 1}} &=O\Par{d^{\frac{5-\omega}{3}}\kappa^{\frac{4\omega - 2}{3}}\log^{\frac{19}{3}}(d\kappa)}.
\end{align*}
Again, we must perform at least one step per ball optimization oracle, so there is an additive $dK\log^3(d\kappa)$ term. Combining these bounds yields the conclusion.
\end{proof}

By instead using a different parameter $C$ as in Corollary~\ref{cor:tradeoff_ball_oracle}, we obtain the following corollary, which interpolates between the two extremes of standard stochastic gradient descent and Theorem~\ref{thm:main_formal}.

\begin{corollary}\label{cor:tradeoff_main}
In the context of Theorem~\ref{thm:main_formal}, for any $C \in [1, \frac{L^2}{\lams^2 r^2}]$ where $\lams, r$ are as defined in \eqref{eq:lams_r_def}, there is an algorithm which solves Problem~\ref{prob:sco} using:
\begin{gather*}
O\Par{d^{\frac 1 3}\kappa^{\frac 2 3}\log^{\frac{13}{3}}\Par{d\kappa}\log\log\Par{d\kappa}\cdot \dQuery + Cd^{\frac 1 3}\kappa^{\frac 2 3}\log^{\frac{28}{3}}(d\kappa)}\text{ depth,}\\
\text{and } O\Par{\Par{d^{\frac 1 3}\kappa^{\frac 2 3}\log^{\frac{10}3}(d\kappa) + \kappa^2\log^{\frac{19}{3}}(d\kappa)} \cdot \tQuery}\\
 + O\Par{C^{2-\omega}\Par{d^{\frac 4 3}\kappa^{\frac 2 3}\log^{\frac {10} 3}(d\kappa)+ d^{\frac{5-\omega}{3}}\kappa^{\frac{4\omega - 2}{3}}\log^{\frac{19}{3}}(d\kappa)}} \text{ work,}
\end{gather*}
where $\omega < 2.372$ \cite{AlmanDWXXZ24} is the matrix multiplication exponent, and $\kappa \defeq \frac{LR}{\eps}$.
\end{corollary} 
\section*{Acknowledgements}

We thank Yair Carmon for helpful conversations during the initial stages of this project. Aaron Sidford was supported in part by a Microsoft Research Faculty Fellowship, NSF CAREER Award CCF-1844855, NSF Grant CCF1955039, and a PayPal research award. Part of this work was conducted while authors were visiting the Simons Institute for the Theory of Computing.

\bibliographystyle{alpha}

\begin{thebibliography}{ADW{\etalchar{+}}24}

\bibitem[ABRW12]{AgarwalBRW12}
Alekh Agarwal, Peter~L. Bartlett, Pradeep Ravikumar, and Martin~J. Wainwright.
\newblock Information-theoretic lower bounds on the oracle complexity of
  stochastic convex optimization.
\newblock {\em {IEEE} Trans. Inf. Theory}, 58(5):3235--3249, 2012.

\bibitem[ACJ{\etalchar{+}}21]{AsiCJJS21}
Hilal Asi, Yair Carmon, Arun Jambulapati, Yujia Jin, and Aaron Sidford.
\newblock Stochastic bias-reduced gradient methods.
\newblock In {\em Advances in Neural Information Processing Systems 34: Annual
  Conference on Neural Information Processing Systems 2021}, pages
  10810--10822, 2021.

\bibitem[ADW{\etalchar{+}}24]{AlmanDWXXZ24}
Josh Alman, Ran Duan, Virginia~Vassilevska Williams, Yinzhan Xu, Zixuan Xu, and
  Renfei Zhou.
\newblock More asymmetry yields faster matrix multiplication.
\newblock {\em CoRR}, abs/2404.16349, 2024.

\bibitem[BJL{\etalchar{+}}19]{BubeckJLLS19}
S{\'{e}}bastien Bubeck, Qijia Jiang, Yin~Tat Lee, Yuanzhi Li, and Aaron
  Sidford.
\newblock Complexity of highly parallel non-smooth convex optimization.
\newblock In {\em Advances in Neural Information Processing Systems 32: Annual
  Conference on Neural Information Processing Systems 2019}, pages
  13900--13909, 2019.

\bibitem[BS18]{BalkanskiS18}
Eric Balkanski and Yaron Singer.
\newblock Parallelization does not accelerate convex optimization: Adaptivity
  lower bounds for non-smooth convex minimization.
\newblock {\em arXiv: 1808.03880}, 2018.

\bibitem[BV04]{bertsimas2004solving}
Dimitris Bertsimas and Santosh Vempala.
\newblock Solving convex programs by random walks.
\newblock {\em Journal of the ACM (JACM)}, 51(4):540--556, 2004.

\bibitem[CGJS23]{ChakrabartyGJS23}
Deeparnab Chakrabarty, Andrei Graur, Haotian Jiang, and Aaron Sidford.
\newblock Parallel submodular function minimization.
\newblock {\em CoRR}, abs/2309.04643, 2023.

\bibitem[CH24]{CarmonH24}
Yair Carmon and Oliver Hinder.
\newblock The price of adaptivity in stochastic convex optimization.
\newblock {\em CoRR}, abs/2402.10898, 2024.

\bibitem[CJJ{\etalchar{+}}20]{CarmonJJJLST20}
Yair Carmon, Arun Jambulapati, Qijia Jiang, Yujia Jin, Yin~Tat Lee, Aaron
  Sidford, and Kevin Tian.
\newblock Acceleration with a ball optimization oracle.
\newblock In {\em Advances in Neural Information Processing Systems 33: Annual
  Conference on Neural Information Processing Systems 2020}, 2020.

\bibitem[CJJ{\etalchar{+}}23]{CarmonJJLLST23}
Yair Carmon, Arun Jambulapati, Yujia Jin, Yin~Tat Lee, Daogao Liu, Aaron
  Sidford, and Kevin Tian.
\newblock Resqueing parallel and private stochastic convex optimization.
\newblock {\em CoRR}, abs/2301.00457, 2023.

\bibitem[DBW12]{DuchiBM12}
John~C Duchi, Peter~L Bartlett, and Martin~J Wainwright.
\newblock Randomized smoothing for stochastic optimization.
\newblock {\em SIAM Journal on Optimization}, 22(2):674--701, 2012.

\bibitem[DG19]{DG19}
Jelena Diakonikolas and Crist{\'{o}}bal Guzm{\'{a}}n.
\newblock Lower bounds for parallel and randomized convex optimization.
\newblock In {\em Conference on Learning Theory, {COLT}}, 2019.

\bibitem[{\relax DLMF}]{DLMF}
{\it NIST Digital Library of Mathematical Functions}.
\newblock http://dlmf.nist.gov/, Release 1.1.8 of 2022-12-15.
\newblock F.~W.~J. Olver, A.~B. {Olde Daalhuis}, D.~W. Lozier, B.~I. Schneider,
  R.~F. Boisvert, C.~W. Clark, B.~R. Miller, B.~V. Saunders, H.~S. Cohl, and
  M.~A. McClain, eds.

\bibitem[GKNS21]{GargKNS21}
Ankit Garg, Robin Kothari, Praneeth Netrapalli, and Suhail Sherif.
\newblock No quantum speedup over gradient descent for non-smooth convex
  optimization.
\newblock In {\em 12th Innovations in Theoretical Computer Science Conference,
  {ITCS} 2021, January 6-8, 2021, Virtual Conference}, volume 185 of {\em
  LIPIcs}, pages 53:1--53:20. Schloss Dagstuhl - Leibniz-Zentrum f{\"{u}}r
  Informatik, 2021.

\bibitem[JLSW20]{jiang2020improved}
Haotian Jiang, Yin~Tat Lee, Zhao Song, and Sam Chiu-wai Wong.
\newblock An improved cutting plane method for convex optimization,
  convex-concave games, and its applications.
\newblock In {\em Proceedings of the 52nd Annual ACM SIGACT Symposium on Theory
  of Computing}, pages 944--953, 2020.

\bibitem[JRT23]{JambulapatiRT23}
Arun Jambulapati, Victor Reis, and Kevin Tian.
\newblock Linear-sized sparsifiers via near-linear time discrepancy theory.
\newblock {\em CoRR}, abs/2305.08434, 2023.

\bibitem[Kha80]{khachiyan1980polynomial}
Leonid~G Khachiyan.
\newblock Polynomial algorithms in linear programming.
\newblock {\em USSR Computational Mathematics and Mathematical Physics},
  20(1):53--72, 1980.

\bibitem[KLL{\etalchar{+}}23]{KelnerLLST23}
Jonathan~A. Kelner, Jerry Li, Allen~X. Liu, Aaron Sidford, and Kevin Tian.
\newblock Semi-random sparse recovery in nearly-linear time.
\newblock In {\em The Thirty Sixth Annual Conference on Learning Theory},
  volume 195 of {\em Proceedings of Machine Learning Research}, pages
  2352--2398. {PMLR}, 2023.

\bibitem[KTE88a]{KTE88}
Leonid~G. Khachiyan, Sergei~Pavlovich Tarasov, and I.~I. Erlikh.
\newblock The method of inscribed ellipsoids.
\newblock {\em Soviet Math. Dokl.}, 37:226--230, 1988.

\bibitem[KTE88b]{khachiyan1988method}
Leonid~G Khachiyan, Sergei~Pavlovich Tarasov, and II~Erlikh.
\newblock The method of inscribed ellipsoids.
\newblock In {\em Soviet Math. Dokl}, volume~37, pages 226--230, 1988.

\bibitem[Lev65]{levin1965algorithm}
Anatoly~Yur'evich Levin.
\newblock An algorithm for minimizing convex functions.
\newblock In {\em Doklady Akademii Nauk}, volume 160, pages 1244--1247. Russian
  Academy of Sciences, 1965.

\bibitem[LSB12]{LacosteJSB12}
Simon Lacoste{-}Julien, Mark Schmidt, and Francis~R. Bach.
\newblock A simpler approach to obtaining an o(1/t) convergence rate for the
  projected stochastic subgradient method.
\newblock {\em CoRR}, abs/1212.2002, 2012.

\bibitem[LSW15]{lee2015faster}
Yin~Tat Lee, Aaron Sidford, and Sam Chiu-wai Wong.
\newblock A faster cutting plane method and its implications for combinatorial
  and convex optimization.
\newblock In {\em 2015 IEEE 56th Annual Symposium on Foundations of Computer
  Science}, pages 1049--1065. IEEE, 2015.

\bibitem[Nem94]{Nem94}
Arkadi Nemirovski.
\newblock On parallel complexity of nonsmooth convex optimization.
\newblock {\em Journal of Complexity}, 10(4):451--463, 1994.

\bibitem[Nes89]{nesterov1989self}
Ju~E Nesterov.
\newblock Self-concordant functions and polynomial-time methods in convex
  programming.
\newblock {\em Report, Central Economic and Mathematic Institute, USSR Acad.
  Sci}, 1989.

\bibitem[New65]{newman1965location}
Donald~J Newman.
\newblock Location of the maximum on unimodal surfaces.
\newblock {\em Journal of the ACM (JACM)}, 12(3):395--398, 1965.

\bibitem[NY83]{NemirovskiY83}
A.\ Nemirovski and D.\~B.\ Yudin.
\newblock {\em Problem Complexity and Method Efficiency in Optimization}.
\newblock Wiley, 1983.

\bibitem[Pan87]{Pan87}
Victor~Y. Pan.
\newblock Complexity of parallel matrix computations.
\newblock {\em Theor. Comput. Sci.}, 54:65--85, 1987.

\bibitem[PR85]{PanR85}
Victor~Y. Pan and John~H. Reif.
\newblock Efficient parallel solution of linear systems.
\newblock In {\em Proceedings of the 17th Annual {ACM} Symposium on Theory of
  Computing}, pages 143--152. {ACM}, 1985.

\bibitem[Sho77]{shor1977cut}
Naum~Z Shor.
\newblock Cut-off method with space extension in convex programming problems.
\newblock {\em Cybernetics}, 13(1):94--96, 1977.

\bibitem[SZ23]{SidfordZ23}
Aaron Sidford and Chenyi Zhang.
\newblock Quantum speedups for stochastic optimization.
\newblock {\em CoRR}, abs/2308.01582, 2023.

\bibitem[Tro15]{Tropp15}
Joel~A. Tropp.
\newblock An introduction to matrix concentration inequalities.
\newblock {\em Found. Trends Mach. Learn.}, 8(1-2):1--230, 2015.

\bibitem[Vai96]{Vaidya96}
Pravin~M. Vaidya.
\newblock A new algorithm for minimizing convex functions over convex sets.
\newblock {\em Math. Program.}, 73:291--341, 1996.

\bibitem[WXXZ23]{WilliamsXXZ23}
Virginia~Vassilevska Williams, Yinzhan Xu, Zixuan Xu, and Renfei Zhou.
\newblock New bounds for matrix multiplication: from alpha to omega.
\newblock {\em CoRR}, abs/2307.07970, 2023.

\bibitem[YN76]{yudin1976informational}
David~B Yudin and Arkadi~S Nemirovskii.
\newblock Informational complexity and efficient methods for the solution of
  convex extremal problems.
\newblock {\em Matekon}, 13(2):22--45, 1976.

\end{thebibliography}
\newcommand{\etalchar}[1]{$^{#1}$}

\newpage

\begin{appendix}
\end{appendix}

\end{document}